\documentclass[12pt,reqno,a4paper]{amsart}

\usepackage{amsfonts}
\usepackage{epsfig}
\usepackage{graphicx}
\usepackage{amsmath}
\usepackage{amssymb}
\usepackage{color}

\newcommand{\Rr}{\mathbb{R}}
\newcommand{\Ss}{\mathbb{S}}

\newcommand{\R}{\mathbb{R}}

\newcommand{\Z}{\mathbb{Z}}

\definecolor{cadmiumgreen}{rgb}{0.0, 0.42, 0.24}

\newcounter{main}

\numberwithin{equation}{section}

\newtheorem{theorem}{Theorem}[section]
\newtheorem{proposition}[theorem]{Proposition}
\newtheorem{lemma}[theorem]{Lemma}

\newtheorem{remark}{Remark}[section]
\newtheorem{example}{Example}[section]
\newtheorem{definition}{Definition}[section]
\newtheorem{maintheorem}{Theorem}
\newtheorem{maincorollary}{Corollary}

\newcommand{\blanksquare}{\,\,\,$\sqcup\!\!\!\!\sqcap$}
{\par}

\title[Shades of hyperbolicity for hamiltonians]{\textbf{Shades of hyperbolicity for hamiltonians}}

\date{\today}

\author[M. Bessa]{M\'{a}rio Bessa}

\address{Departamento de Matem\'atica, Universidade da Beira Interior, Rua Marqu\^es d'\'Avila e Bolama,
  6201-001 Covilh\~a,
Portugal.}
\email{bessa@fc.up.pt}

\author[J. Rocha]{Jorge Rocha}
\address{Departamento de Matem\'atica, Universidade do Porto,
Rua do Campo Alegre, 687,
4169-007 Porto, Portugal}
\email{jrocha@fc.up.pt}

\author[M. J. Torres]{Maria Joana Torres}
\address{CMAT, Departamento de Matem\'atica e Aplica\c{c}\~{o}es, Universidade do Minho,
Campus de Gualtar,
4700-057 Braga, Portugal}
\email{jtorres@math.uminho.pt}

\begin{document}

\maketitle

\begin{abstract}

We prove that a Hamiltonian system $H\in C^2(M,\mathbb{R})$ is globally hyperbolic if any of the following statements holds: $H$ is robustly topologically stable; $H$ is stably shadowable; $H$ is stably expansive; and $H$ has the stable weak specification property.  Moreover, we prove that, for a $C^2$-generic Hamiltonian $H$, the union of the partially hyperbolic regular energy hypersurfaces and the closed elliptic orbits, forms a dense subset of $M$. As a consequence, any robustly transitive regular energy hypersurface of a $C^2$-Hamiltonian is partially hyperbolic. Finally, we prove that stably weakly shadowable regular energy hypersurfaces are partially hyperbolic.
\end{abstract}

\bigskip

{\footnotesize\textbf{Keywords:} \emph{Hamiltonian vector field, structural stability, dominated splitting, elliptic closed orbits.}

\textbf{2000 Mathematics Subject Classification:} Primary: 37J10,  37D30; Secondary: 70H05.}

\bigskip

\begin{section}{Introduction and a tour along the main results}
The Hamiltonian functions $H\colon M\rightarrow \mathbb R$ on symplectic $2n$-dimensional manifolds $M$ form a central subclass of all continuous-time dynamical systems generated by the $2n$ differential equations known as the \emph{Hamilton equations} $\dot{q}_i=\partial_{p_i} H$ and $\dot{p}_i=-\partial_{q_i} H$, where $(q_i,p_i)\in M$ and $i=1,..,n$. Their relevance follows from the vast range of applications throughout several branches of mathematics. Actually, the laws of mathematical physics are mostly expressed in terms of differential equations, and a well studied and successful subclass of these differential equations, whose solutions keep invariant a given symplectic form, are the Hamiltonian equations (see \cite{AbM,Arn}).

Stable properties of Hamiltonians, i.e. those properties that are shared by slightly perturbed systems of these continuous-time systems are a fundamental problem in describing the dynamics in the space of all Hamiltonians. Among them the structurally stable property, introduced in the mid 1930s by Andronov and Pontrjagin, plays a fundamental role. Roughly speaking it means that under small perturbations the dynamics are topologically equivalent.
Smale's program in the early 1960s aimed to prove the (topological) genericity of structurally stable systems. Although Smale's program was proved to be wrong one decade later, it played a fundamental role in the theory of dynamical systems and raised the problem of characterizing
structural stability as being essentially equivalent to uniform hyperbolicity.
In \cite{BRT} the authors characterized the structurally stable Hamiltonian systems as the Anosov ones.
The characterization of structurally stable systems, using topological and geometric dynamical properties, has been one of the main object of interest in the global
qualitative theory of dynamical systems, in the last 40 years.
The purpose of one of our main results (see \textit{Theorem~\ref{mainth1}}) is, thus,
to give a characterization of the structurally stable Hamiltonians by making use of the notions of topological stability, shadowing, expansiveness and specification. These properties although they seem quite different they are intertconnected with uniformly hyperbolic demeanor.

Anosov systems, and thus, structurally stable Hamiltonians, are topologically stable, expansive and satisfy the shadowing property. But the converse is not true. There are examples of systems far from the structurally stable ones that are
topologically stable\footnote{It is immediate that the topological stability is invariant by conjugacy. Moreover, by a result of Gogolev (see~\cite{Go}) the existence of a conjugacy between two maps where one of which is Anosov is not sufficient to guarantee that the other is Anosov.}, that are expansive (see Example~\ref{pedro}), and that satisfy the shadowing property (see ~\cite{YY} and also the symplectic curious example of existence of shadowing of transition chains of several invariant
tori alternating with Birkhoff zones of instability~\cite{GR}).

Therefore, the problem on the relationship between structural stability and topological and geometric properties of the system is not trivial. The passage to $C^2$-interiors of sets of Hamiltonians that have these properties developed in the present work became one effective approach to the solution of this problem.
A property is $C^2$-stable if once the property holds for a $C^2$-Hamiltonian it also holds for every $C^2$-nearby Hamiltonian.
Let us explain the dynamical properties that we deal with and state Theorem~\ref{mainth1}.

The \emph{shadowing} of a given set of points which are an ``almost orbit" of some system by a true orbit appears in many branches of dynamical systems and is quite often related with hyperbolicity. Actually, the computational estimates, fitted with a certain error of orbits, are meaningless if they cannot be realized by genuine orbits of the given system, and, in this sense, are mere pixel inaccuracy typical of the computational framework. Despite the fact that shadowable systems can be non hyperbolic the stability of shadowing goes a long way when proving hyperbolicity. We refer the work developed in ~\cite{B2,F2,LS,P1,S2,WGW} both for flows and diffeomorphisms and in dissipative and conservative contexts. Here we will prove that stability of shadowing is equivalent to hyperbolicity.

The notion of \emph{topological stability} (see the precise definition in \S \ref{topsta}), grew parallel to the theory of structural stability, and was first introduced by Walters in (\cite{W}) when proving that Anosov diffeomorphisms are topologically stable. Then, Nitecki proved that topological stability was a necessary condition to get Axiom A plus strong transversality (see \cite{Ni}). In the late 1970's (\cite{Ro2}), Robinson proved that Morse-Smale flows are topologically stable, and in the mid 1980's, Hurley obtained necessary conditions for topological stability (see \cite{H,H2,H3}). Here we generalize the results for flows in ~\cite{BR2,MSS2}) to the Hamiltonian context (see also \cite{BR3}) which says that robust topologically stable vector fields are hyperbolic.

A dynamical system is said to be \emph{expansive} when, in brief terms, if two points stay near for future and past iterates, then they must be equal. We can say, in a general scope, that the system has sensitivity to the initial conditions, because, when different, two points must be separated by forward or backward iteration. This notion was first developed in the 1950's (see~\cite{Wi}) and, in the flow context, initialized by the studies of Bowen and  Walters (see~\cite{BW}). In this paper we generalize the recent results in ~\cite{F2,MSS,Se} by proving that the stability of expansiveness is equivalent to hyperbolicity.

The concept of \emph{weak specification} although intricate (see \S\ref{spec}) is quite well summarized in simple words, for diffeomorphisms, in ~\cite[page 193]{DGS}: \emph{``The weak specification means that whenever there are two pieces of orbits $\{f^n(x_1)\colon a_1\leq n \leq b_1\}$ and $\{f^n(x_2)\colon a_2\leq n \leq b_2\}$, they may be approximated up to $\epsilon$ by one single orbit - the
orbit of $x$ - provided that the time for \emph{switching} from the
first piece of orbit to the second (namely $a_2-b_1$) and the
time for \emph{switching back} (namely $p - (b_2-a_1)$)\footnote{The number $p$ is any number greater or equal than $K(\epsilon)+b_2-a_1$.} are larger than
$K(\epsilon)$, this number $K(\epsilon)$ being independent of the pieces of orbit,
and in particular independent of their length."}\\
Several authors obtained hyperbolicity from the hypothesis that the system has the stable specification property (see ~\cite{ASS,SSY}). We point out that their arguments were supported on a change of index argument in the hyperbolic closed orbits. However, in the symplectic setting such situation is impossible to happen because the index is constant and equal to $n-1$. Here we obtain similar results for Hamiltonians using a new symplectic approach.

Summarizing, in Theorem 1, we prove that a Hamiltonian is globally hyperbolic (Anosov) if any of the following properties is stable: topological stability, shadowing, expansiveness, and specification. Thus, we call these properties \textit{shades of hyperbolicity}.
The proof is a direct consequence of the following. First, we refer that it was recently proved that a star Hamiltonian is Anosov (see \textit{Theorem~\ref{thm2}}, proved in
~\cite{BFR} for $n=2$ and generalized by the authors in~\cite{BRT} for $n \geq 2$).
We recall that a Hamiltonian is a {\em star Hamiltonian} if the property of having all periodic orbits of hyperbolic type is stable.
Finally, we prove, in \textit{Theorem~\ref{thm3}}, that the stability of any of the shade properties just described implies that the Hamiltonian is a star system.

Dynamical systems exhibiting dense orbits are called \emph{transitive}. Informally speaking this means that the whole manifold is indecomposable. Those systems for which this property remains valid for any perturbation are called \emph{robustly transitive}. Since the pioneering work of Ma\~n\'e (see~\cite{M}) several other results were obtained guaranteeing that robust transitivity, with respect to $C^1$-topology, implies a certain form of hyperbolicity (see e.g. ~\cite{AM,BR1,BDP,BGV,HT,V}). Here, and as a direct consequence a dichotomy that we discuss in the sequel, we shall prove that robustly transitive Hamiltonians are partially hyperbolic. Moreover, we shall prove that stably ergodic Hamiltonians are partially hyperbolic.

The shadowing property in the weak sense first appear in the paper by Corless and Pilyugin (see ~\cite{CP}) when related to the genericity (with respect to the $C^0$-topology) of shadowing among homeomorphisms. In simple terms \emph{weak shadowing} allows that the ``almost orbits" may be approximated by true orbits if one forgets the time parameterization and consider only the distance between the orbit and the ``almost orbit" as two subsets in the manifold. There exist dynamical systems without the weak shadowing (see ~\cite[Example 2.12]{P}) and dynamical systems satisfying the weak shadowing but not the shadowing property (\cite[Example 2.13]{P}. In the present paper we generalize a recent result (see ~\cite{BLV,BVa}) for the setting of Hamiltonians. More precisely, we obtain partial hyperbolicity under the hypothesis of stable weak shadowing (see \textit{Theorem~\ref{thm6}}).

Generic properties of Hamiltonians, i.e. those properties that are shared by a countable intersection of open and dense subsets of these continuous-time systems, are thus of great importance and interest since they give us the typical behavior, in Baire's category sense, that one could expect from a wide class of systems (cf.~\cite{BFR, BFR2, MBJLD, MBJLD2, BRT, N, R0}). There are obvious constraints to the amount of information we can obtain from a certain cluster of generic systems. Howsoever, it is of great importance and utility to learn that a given system can be slightly perturbed
in order to obtain a system whose global and local dynamics we understand quite well.

Questions concerning the generic behaviour of Hamiltonians were first raised by Robinson (see \cite{R0}). One of our main generic results (see \textit{Theorem~\ref{thm5}}) is the generalization of a result stated in~\cite{N} by Newhouse and proved in~\cite{MBJLD2} for the $4$-dimensional context (or a weak $2n$-dimensional version): $C^2$-generic Hamiltonians are of hyperbolic type or else exhibit dense $1$-elliptic closed orbits. We prove that
most Hamiltonians, from the generic viewpoint, have \textbf{only} two types of well-di\-ffe\-ren\-ti\-a\-ted behavior: partial hyperbolicity (chaotic type, cf. ~\cite{BDV}) or else lots of elliptic closed orbits (KAM type, cf. \cite{Y}).
Of course that, in the 4-dimensional case,  $1$-elliptic closed orbits are totally elliptic and, moreover, partial hyperbolicity is actually hyperbolicity. See (\cite{MBJLD2}) for a detailed proof of this result. It is still an open and quite interesting question to know if these type of results hold for mechanic systems (see ~\cite[\S 8]{MBJLD3}).

With respect to the discrete-time case, in \cite{N}, it was obtained that $C^1$-generic symplectomorphisms in surfaces are Anosov (uniformly hyperbolic) or else the elliptic points are dense. Long after Newhouse's proof, Arnaud (see ~\cite{Ar2002}) proved the 4-dimensional version of this result, namely that  $C^1$-generic symplectomophisms are Anosov, partially hyperbolic or have dense elliptic periodic points. Finally, Saghin and Xia (\cite{SX}) proved the same result but for any dimension completing the program for the discrete case. In this paper, and among other issues, we developed, in \textit{Theorem~\ref{thm4}}, the approach for the continuous-time case of Saghin-Xia's theorem.

We emphasize that in this paper we are restricted to the $C^2$-topology for Hamiltonians and to the $C^1$-topology for the associated vector field, since our proofs use several technical results which are proved for this topology (see the perturbation results in \S \ref{hamilperturbsection}).

\end{section}

\bigskip

\begin{section}{Hamiltonian systems}

\begin{subsection}{The Hamiltonian framework}

Let $(M,\omega)$ be a symplectic manifold, where $M$ is a $2n$-dimensional ($n \geq 2$), compact, boundaryless, connected and
smooth Riemannian manifold, endowed with a symplectic form $\omega$.
A \textit{Hamiltonian} is a real-valued $C^r$ function on $M$,  $2\leq r\leq \infty$. We denote by $C^r(M,\mathbb{R})$
the set of $C^r$-Hamiltonians on $M$. From now on, we shall restrict to the $C^2$-topology, and thus we set $r=2$.
Associated with $H$, we have the \textit{Hamiltonian vector field} $X_H$ defined by $\omega(X_H(p),u)=\nabla H_p(u)$, for all $u \in T_p M$,
which generates the
Hamiltonian flow $X_H^t$. Observe that $H$ is $C^2$ if and only if $X_H$ is $C^1$ and that, since $H$ is continuous and $M$ is compact, $Sing(X_H)\neq\emptyset$, where $Sing(X_H)$ denotes the singularities of $X_H$ or, in other words, the critical points of $H$ or the equilibria of $X_H^t$. Let $\mathcal{R}(H)=M\setminus Sing(X_H)$ stands for the set of regular points.

A scalar $e\in H(M)\subset \mathbb{R}$ is called an \textit{energy} of $H$ and the pair $(H,e)$ is called \textit{Hamiltonian level}.
An \textit{energy hypersurface} $\mathcal{E}_{H,e}$ is a connected component of $H^{-1}(\left\{e\right\})$, called \textit{energy level set}, and it is \textit{regular} if it does not contain singularities.
Observe that a regular energy hypersurface is a $X_H^t$-invariant, compact and $(2n-1)$-dimensional manifold.
The energy level set $H^{-1}(\left\{e\right\})$ is said \textit{regular} if any energy hypersurface of
$H^{-1}(\left\{e\right\})$ is regular.
If $H^{-1}(\left\{e\right\})$ is regular, then $H^{-1}(\left\{e\right\})$ is the union of a finite number of energy hypersurfaces.
Finally, a Hamiltonian level  $(H,e)$ is said \textit{regular} if the energy level set $H^{-1}(\left\{e\right\})$ is regular.

A \textit{Hamiltonian system} is a triple $(H,e, \mathcal{E}_{H,e})$, where $H$ is a Ha\-mil\-ton\-ian, $e$ is an energy and $\mathcal{E}_{H,e}$ is a regular connected component of $H^{-1}(\{e\})$.

Fixing a small neighbourhood $\mathcal{W}$ of a regular $\mathcal{E}_{H,e}$, there exist a small neighbourhood $\mathcal{U}$ of $H$ and $\epsilon>0$ such that, for all $\tilde{H} \in \mathcal{U}$ and $\tilde{e} \in (e-\epsilon,e+\epsilon)$, $\tilde{H}^{-1}(\{\tilde{e}\})\cap \mathcal{W}=\mathcal{E}_{\tilde{H},\tilde{e}}$. We call $\mathcal{E}_{\tilde{H},\tilde{e}}$ the \textit{analytic continuation} of $\mathcal{E}_{H,e}$.

In the space of Hamiltonian systems we consider the topology generated by a fundamental systems of neighbourhoods.
Given a Hamiltonian system $(H,e, \mathcal{E}_{H,e})$ we say that $\mathcal{V}$ is a {\em neighbourhood} of $(H,e, \mathcal{E}_{H,e})$ if there exist a small neighbourhood $\mathcal{U}$ of $H$ and $\epsilon>0$ such that for all $\tilde{H} \in \mathcal{U}$ and $\tilde{e} \in (e-\epsilon,e+\epsilon)$ one has that the analytic continuation $\mathcal{E}_{\tilde{H},\tilde{e}}$ of $\mathcal{E}_{H,e}$ is well-defined.

Since the symplectic form $\omega$ is non-degenerate, given $H\in C^2(M,\mathbb{R})$ and $p\in M$, we know that $\nabla H_p=0$ is equivalent to $X_H(p)=0$. Therefore, the extreme values of a Hamiltonian $H$ are exactly the singularities of the associated Hamiltonian vector field $X_H$.

Given a Hamiltonian level $(H,e)$, let $\Omega(H|_{\mathcal{E}_{H,e}})$ be the set of non-wandering points of $H$ on the energy hypersurface $\mathcal{E}_{H,e}$, that is, the points $x\in\mathcal{E}_{H,e}$ such that, for every neighborhood $U$ of $x$ in $\mathcal{E}_{H,e}$, there is $\tau>0$ such that $X^{\tau}_H(U)\cap U\neq\emptyset$.

By \textit{Liouville's Theorem}, the symplectic manifold $(M,\omega)$ is also a volume manifold (see, for example, \cite{AbM}).
This means that the $2n$-form $\omega^n=\omega\wedge\overset{n}{...}\wedge\omega$ (wedging $n$-times) is a volume form and induces a measure $\mu$ on $M$, which is called the Lebesgue measure associated to $\omega^n$.
Notice that the measure $\mu$ on $M$ is invariant by the Hamiltonian flow.
So, given a re\-gu\-lar Hamiltonian level $(H,e)$, we induce a volume form $\omega_{\mathcal{E}_{H,e}}$ on each energy hypersurface $\mathcal{E}_{H,e}\subset H^{-1}(\left\{e\right\})$, where for all $p\in \mathcal{E}_{H,e}$:
\begin{align}
\omega_{\mathcal{E}_{H,e}}:\;&T_p\mathcal{E}_{H,e} \times T_p\mathcal{E}_{H,e}\times T_p\mathcal{E}_{H,e} \longrightarrow \mathbb{R}\nonumber\\
& (u,v,w)\longmapsto \omega^n(\nabla H_p,u,v,w)\nonumber
\end{align}

The volume form $\omega_{\mathcal{E}_{H,e}}$ is $X_H^t$-invariant.
Hence, it induces an invariant volume measure $\mu_{\mathcal{E}_{H,e}}$ on $\mathcal{E}_{H,e}$ which is a finite measure, since any energy hypersurface is compact. Observe that, under these conditions, we have that $\mu_{\mathcal{E}_{H,e}}$-a.e. $x\in \mathcal{E}_{H,e}$ is recurrent, by the \textit{Poincar\'{e} Recurrence Theorem}.
\end{subsection}

%%%%%%%%%%%%%%%%%%%%%%%%%%%%%%%%%%%%%%%%%%%%%%%%%%%%%%%%%%%%%%%%%%%%%%%%%%%%%%%%%%%%%%%%%%%%%%%%%%%%%%%%%%%%%%%%%%%%%%%%%%%%%%%%%%%%%%%%%%%%%%%%%%%%%%%%

\begin{subsection}
{Transversal linear Poincar\'{e} flow and hyperbolicity}\label{translinpoinsec}

Let us begin with the definition of the \textit{transversal linear Poincar\'{e} flow}. After, we state some results using this linear flow. Consider a Hamiltonian vector field $X_H$ and a regular point $x$ in $M$ and let $e=H(x)$.
Define $\mathcal{N}_x:=N_x\cap T_xH^{-1}(\left\{e\right\})$, where
$N_x=(\mathbb{R} X_H(x))^\perp$ is the normal fiber at $x$, $\mathbb{R} X_H(x)$ stands for the flow direction at $x$, and
$T_xH^{-1}(\left\{e\right\})=Ker\: \nabla H_x$ is the tangent space to the energy level set.
Thus, $\mathcal{N}_x$ is a ($2n-2$)-dimensional bundle.

\begin{definition}\label{TLPF}
The \textbf{transversal linear Poincar\'{e} flow} associated to $H$ is given by
\begin{align}
\Phi_H^t(x): &\ \ \mathcal{N}_{x}\rightarrow \mathcal{N}_{X_H^t(x)}\nonumber\\
&\ \ v\mapsto \Pi_{X_H^t(x)}\circ D{X_H}^t_x(v),\nonumber
\end{align}
where $\Pi_{X_H^t(x)}: T_{X_H^t(x)}M\rightarrow \mathcal{N}_{X_H^t(x)}$ denotes the canonical orthogonal projection.
\end{definition}

Observe that $\mathcal{N}_x$ is $\Phi^t_{H}(x)$-invariant. It is well-known (see e.g. ~\cite{AbM}) that, given a regular point $x\in\mathcal{E}_{H,e}$, then $\Phi_H^t(x)$ is a linear symplectomorphism for the symplectic form $\omega_{\mathcal{E}_{H,e}}$, that is, $\omega_{\mathcal{E}_{H,e}}(u,v)=\omega_{\mathcal{E}_{H,e}}(\Phi_H^t(x)\:u,\Phi_H^t(x)\:v)$, for any $u,v\in \mathcal{N}_{x}$.

For any symplectomorphism, in particular for $\Phi_H^t(x)$, we have the following result.

\begin{theorem} (Symplectic eigenvalue theorem, \cite{AbM})\label{sympeigen}
Let $f$ be a symplectomorphism in $M$, $p\in M$ and $\sigma$ an eigenvalue of $Df_p$ of multiplicity $m$. Then $1/\sigma$ is an eigenvalue of $Df_p$ of multiplicity $m$. If $\sigma$ is non real, then $\overline{\sigma}$ and $1/\overline{\sigma}$ are also eigenvalues of $Df_p$. Moreover, the multiplicity of the eigenvalues $+1$ and $-1$, if they occur, is even.
\end{theorem}

The proof of the following result can be found in~\cite[Section 2.3]{MBJLD}.

\begin{lemma}\label{hyperbolictranspoinc}
Take a Hamiltonian $H\in C^2(M,\mathbb{R})$ and let $\Lambda$ be an $X_H^t$-invariant, regular and compact subset of $M$. Then, $\Lambda$ is uniformly hyperbolic for $X_H^t$ if and only if the induced transversal linear Poincar\'{e} flow $\Phi_H^t$ is uniformly hyperbolic on $\Lambda$.
\end{lemma}
So, we can define a \textsl{uniformly hyperbolic set} as follows.

\begin{definition}\label{Phd}
Let $H\in C^2(M,\mathbb{R})$.
An $X_H^t$-invariant, compact and regular set $\Lambda\subset M$ is \textbf{uniformly hyperbolic} if $\mathcal{N}_{\Lambda}$ admits a $\Phi_H^t$-invariant splitting $\mathcal{N}_{\Lambda}^s\oplus \mathcal{N}_{\Lambda}^u$ such that there is $\ell>0$ satisfying
\begin{center}
$\|\Phi_H^{\ell}(x)|_{\mathcal{N}^s_x}\|\leq \dfrac{1}{2}$ and $\|\Phi_H^{-{\ell}}(X_H^{\ell}(x))|_{\mathcal{N}^u_{X_H^{\ell}(x)}}\|\leq\dfrac{1}{2}$, for any $x\in\Lambda$.
\end{center}
\end{definition}
We remark that the constant $\frac{1}{2}$ can be replaced by any constant $\theta\in(0,1)$ with the detriment of changing the value of $\ell$.

Given $x\in \mathcal{R}(H)$, we say that $x$ is a \textit{periodic point} of the Hamiltonian $H$ if $X^t_H(x)=x$ for some $t$. The smallest $t_0>0$ satisfying the condition above is called \emph{period} of $x$; in this case, we say that the orbit of $x$ is a \textit{closed orbit} of period $t_0$.
Accordingly with Definition~\ref{Phd},
a periodic point $x$ is \textit{hyperbolic} if
there exists a splitting of the normal subbundle $\mathcal{N}$ along the orbit of $x$ that satisfies the condition above.

Now, we state the definition of \textsl{dominated splitting}, by using the transversal linear Poincar\'{e} flow.

\begin{definition}\label{DDPoinc}
Take $H\in C^2(M,\mathbb{R})$ and let $\Lambda$ be a compact, $X_H^t$-invariant and regular subset of $M$.
Consider a $\Phi_H^t$-invariant splitting $\mathcal{N}=\mathcal{N}^1\oplus\cdots\oplus \mathcal{N}^{k}$ over $\Lambda$, for $1\leq k\leq 2n-2$, such that all the subbundles have constant dimension.
This splitting is \textbf{dominated} if there exists ${\ell}>0$ such that, for any $0\leq i<j\leq k$,
\begin{center}
${\|\Phi_H^{\ell}(x)|_{{\mathcal{N}}_x^i}\|}\cdot{\|\Phi_H^{-\ell}(X_H^{\ell}(x))|_{{\mathcal{N}}_{X_H^{\ell}(x)}^j}\|}\leq\dfrac{1}{2}, \: \: \forall \: x\in \Lambda.$
\end{center}\end{definition}

Finally, we state the definition of \textsl{partial hyperbolicity}, of the transversal linear Poincar\'{e} flow.

\begin{definition}\label{PH}
Take $H\in C^2(M,\mathbb{R})$ and let $\Lambda$ be a compact, $X_H^t$-invariant and regular subset of $M$.
Consider a $\Phi_H^t$-invariant splitting $\mathcal{N}=\mathcal{N}^u\oplus \mathcal{N}^c \oplus \mathcal{N}^{s}$ over $\Lambda$ such that all the subbundles have constant dimension and at least two of them are non-trivial.
This splitting is \textbf{partially hyperbolic} if there exists ${\ell}>0$ such that,
\begin{enumerate}
\item $\mathcal{N}^u$ is uniformly hyperbolic and expanding;
\item $\mathcal{N}^s$ is uniformly hyperbolic and contracting and
\item $\mathcal{N}^u$ $\ell$-dominates $\mathcal{N}^c$ and $\mathcal{N}^c$ $\ell$-dominates $\mathcal{N}^s$.
\end{enumerate}
\end{definition}

In general, along this paper, we consider these three structures defined in a set $\Lambda$ which is the whole energy level.

\begin{remark}\label{Mane}
It was proved in \cite{BV} that, in the symplectic world, the existence of a do\-mi\-na\-ted splitting implies partial hyperbolicity. More precisely, If $\mathcal{N}^u\oplus \mathcal{N}^2$ is a dominated splitting, with $\dim \mathcal{N}^u\leq \dim\mathcal{N}^2$, then $\mathcal{N}^2$ splits invariantly as $\mathcal{N}^2 = \mathcal{N}^c \oplus \mathcal{N}^s$, with $\dim \mathcal{N}^s = \dim \mathcal{N}^u$. Furthermore, the splitting $\mathcal{N}^u \oplus\mathcal{N}^c  \oplus \mathcal{N}^s$ is dominated, $\mathcal{N}^u$ is uniformly expanding, and $\mathcal{N}^s$ is uniformly contracting. In conclusion, $\mathcal{N}^u \oplus\mathcal{N}^c  \oplus \mathcal{N}^s$ is partially hyperbolic.
\end{remark}

\end{subsection}

\end{section}

\bigskip

\begin{section}{Shade properties: topological stability, shadowing, weak shadowing, expansiveness, specification, transitivity and ergodicity}

In this section we describe the dynamical properties that we shall deal with in the sequel.

\begin{subsection}{Topological stability}\label{topsta}
Let $(H,e, \mathcal{E}_{H,e})$ and $(\tilde{H},\tilde{e}, \mathcal{E}_{\tilde{H},\tilde{e}})$ be Hamiltonians systems; we say that $(\tilde{H},\tilde{e}, \mathcal{E}_{\tilde{H},\tilde{e}})$ is {\em semiconjugated} to $(H,e, \mathcal{E}_{H,e})$ if there exist a continuous and onto map $h: \mathcal{E}_{\tilde{H},\tilde{e}} \rightarrow \mathcal{E}_{H,e}$ and a continuous real map $\tau: \mathcal{E}_{\tilde{H},\tilde{e}}  \times \Rr \rightarrow \Rr$ such that:

\smallskip
\begin{enumerate}
\item[$\mbox{(a)}$] for any $x \in \mathcal{E}_{\tilde{H},\tilde{e}}$, $\tau_x: \Rr \rightarrow \Rr$ is an orientation preserving homeomorphism
where $\tau(x,0)=0$ and

\item[$\mbox{(b)}$] for all $x \in \mathcal{E}_{\tilde{H},\tilde{e}}$ and $t \in \Rr$ we have $h(X_{\tilde{H}}^t(x))=X_{H}^{\tau(x,t)}(h(x)).$
\end{enumerate}

\smallskip

 We say that the Hamiltonian system $(H,e, \mathcal{E}_{H,e})$ is {\em topologically stable}
if for any $\epsilon>0$, there exists
$\delta>0$ such that for any Hamiltonian system $(\tilde{H},\tilde{e}, \mathcal{E}_{\tilde{H},\tilde{e}})$
such that $\tilde{H}$ is $\delta$-$C^1$-close to $H$ and $\tilde{e} \in (e-\delta,e+\delta)$,
there exists a semiconjugacy from  $(\tilde{H},\tilde{e}, \mathcal{E}_{\tilde{H},\tilde{e}})$ to $(H,e, \mathcal{E}_{H,e})$,
i.e., there exists  $h: \mathcal{E}_{\tilde{H},\tilde{e}} \rightarrow \mathcal{E}_{H,e}$ and
$\tau: \mathcal{E}_{\tilde{H},\tilde{e}}  \times \Rr \rightarrow \Rr$ satisfying $\mbox{(a)}$ and $\mbox{(b)}$ above, and
$d(h(x),x) < \epsilon$ for all $x \in \mathcal{E}_{\tilde{H},\tilde{e}}$.
Observe that the notion of topologically stability does not define an equivalence relation.
Furthermore, the set of systems semi-conjugated to a given Hamiltonian system may not be an open set.
This motivates the following definition.
We say that $(H,e, \mathcal{E}_{H,e})$ is {\em robustly topologically stable}
if
there exists a neighbourhood $\mathcal{V}$ of $(H,e, \mathcal{E}_{H,e})$ such that any $(\tilde{H},\tilde{e}, \mathcal{E}_{\tilde{H},\tilde{e}}) \in \mathcal{V}$ is topologically stable.

\end{subsection}

\begin{subsection}{The shadowing property}
Let $(H,e, \mathcal{E}_{H,e})$ be a Hamiltonian system. Let us fix real numbers $\delta, T>0$. We say that a pair of
sequences $((x_i), (t_i))_{i \in \Z}$ $(x_i \in \mathcal{E}_{H,e}, t_i \in \R, t_i \geq T)$ is a
$(\delta,T)$-{\em pseudo-orbit} of $H$ if $d(X_H^{t_i}(x_i),x_{i+1}) < \delta$ for all $i \in \Z$.
For the sequence $(t_i)_{i \in \Z}$ we write
$S(n)=t_0+t_1+\ldots+t_{n-1}$ if $n \geq 0$, and
$S(n)=-(t_n+\ldots+t_{-2}+t_{-1})$ if $n < 0$, where $S(0)=t_0+t_{-1}=0.$
Let $x_0 \star t$ denote a point on a $(\delta,T)$-chain $t$ units time from $x_0$.
More precisely, for $t \in \R$,
$$x_0 \star t=X_H^{t-S(i)}(x_i) \hspace{0.4cm} {\text{\rm if}} \hspace{0.4cm} S(i) \leq t <S(i+1).$$
Simple examples show that, in the case of a flow, it is unnatural to require in the definition of shadowing the closeness
of points of a pseudo-orbit and its shadowing exact orbit corresponding to the same instants of time, as it is posed in the shadowing problem for diffeomorphisms. We need to reparametrize the exact shadowing orbit. By $\mbox{Rep}$ we denote the set of all increasing homemorphisms $\alpha: \R \rightarrow \R$,
called {\em reparametrizations}, satisfying $\alpha(0)=0$. Fixing $\epsilon>0$, we define the set
$\mbox{Rep}(\epsilon)=\left\{\alpha \in \mbox{Rep}: \left|\frac{\alpha(t)}{t}-1 \right|<\epsilon, \, t \in \R \right\}.$
When we choose a reparametrization $\alpha$ in the previous set, we want $\alpha$ to be taken arbitrarily close to identity.
A $(\delta,T)$-pseudo-orbit $((x_i), (t_i))_{i \in \Z}$ is $\epsilon$-{\em shadowed} by some orbit of $H$ if there is $z \in \mathcal{E}_{H,e}$ and a reparametrization $\alpha \in \mbox{Rep}(\epsilon)$ such that $d(X_H^{\alpha(t)}(z),x_0 \star t)<\epsilon$, for every $t \in \R$.

The Hamiltonian system $(H,e, \mathcal{E}_{H,e})$ is said to have the {\em shadowing property} if, for any $\epsilon>0$ there exist $\delta, T>0$ such that
any $(\delta,T)$-pseudo-orbit $((x_i), (t_i))_{i \in \Z}$ is $\epsilon$-{shadowed} by some orbit of $H$.

We say that the Hamiltonian system $(H,e, \mathcal{E}_{H,e})$ is {\em stably shadowable} if
there exists a neighbourhood $\mathcal{V}$ of $(H,e, \mathcal{E}_{H,e})$ such that any $(\tilde{H},\tilde{e}, \mathcal{E}_{\tilde{H},\tilde{e}}) \in \mathcal{V}$ has the shadowing property.

\end{subsection}

\begin{subsection}{The weak shadowing property}
We recall the following definition of weakly shadowable systems that was introduced in \cite{CP} in connection with the problem of genericity of shadowing. Given a Hamiltonian system  $(H,e, \mathcal{E}_{H,e})$ and $\delta, T>0$, a
$(\delta,T)$-pseudo-orbit $((x_i), (t_i))_{i \in \Z}$ is {\em weakly} $\epsilon$-{\em shadowed} by some orbit of $H$
if there exists $z \in \mathcal{E}_{H,e}$ such that $\{x_i\}_{i \in \Z} \subset B_\epsilon(\mathcal{O}(z))$.

The Hamiltonian system $(H,e, \mathcal{E}_{H,e})$ is said to have the {\em weak shadowing property} if, for any $\epsilon>0$ there exist $\delta, T>0$ such that
any $(\delta,T)$-pseudo-orbit $((x_i), (t_i))_{i \in \Z}$ is weakly $\epsilon$-shadowed by some orbit of $H$.

We say that the Hamiltonian system $(H,e, \mathcal{E}_{H,e})$ is {\em stably weakly shadowable} if
there exists a neighbourhood $\mathcal{V}$ of $(H,e, \mathcal{E}_{H,e})$ such that any $(\tilde{H},\tilde{e}, \mathcal{E}_{\tilde{H},\tilde{e}}) \in \mathcal{V}$ has the weak shadowing property.

\end{subsection}

\begin{subsection}{The expansiveness property}
Let $(H,e, \mathcal{E}_{H,e})$ be a Hamiltonian system. We say that $(H,e, \mathcal{E}_{H,e})$ is {\em expansive} if, for any $\epsilon>0$, there
is $\delta>0$ such that if $x,y \in \mathcal{E}_{H,e}$ satisfy $d(X^t_H(x),X^{\alpha(t)}_H(y)) \leq \delta$, for any $t \in \R$
and for some continuous map $\alpha: \R \rightarrow \R$ with $\alpha(0)=0$, then
$y=X^s_H(x)$, where $|s| \leq \epsilon$.

This definition asserts that any two points whose orbits remain indistinguishable, up to any continuous time displacement,
must be in the same orbit.

Observe that the reparametrization $\alpha$ is not assumed to be close to identity and that the expansiveness property does not
depend on the choice of the metric on $M$.

We say that the Hamiltonian system $(H,e, \mathcal{E}_{H,e})$ is {\em stably expansive} if
there exists a neighbourhood $\mathcal{V}$ of $(H,e, \mathcal{E}_{H,e})$ such that any $(\tilde{H},\tilde{e}, \mathcal{E}_{\tilde{H},\tilde{e}}) \in \mathcal{V}$ has the expansiveness property.

The next example, point to us by Pedro Duarte, shows that expansiveness can coexist with elliptic orbits.

\begin{example}\label{pedro}
Consider the Hamiltonian with 1-degree of freedom given by $H(x,y)=x^3-3xy^2$. The associated Hamiltonian vector field if $X_H(x,y)=(-6xy,3y^2-3x^2)$ (for more details see ~\cite[Appendix 7]{Arn}). The origin is a degenerated singularity of the vector field. However, our system exhibits a symmetry; the rotation by $\frac{2\pi}{3}$ centered in $(0,0)$ keeps invariant the phase portrait (see Figure 1). Hence, we can keep the same dynamics as the one in the figure and turn the origin into an elliptic fixed point inducing a rotation of angle $\frac{2\pi}{3}$. Observe that, locally, the system is expansive despite the fact that we are in the presence of an elliptic point.
\end{example}

\begin{figure}[h!]
  \centering
     \includegraphics[height=5cm]{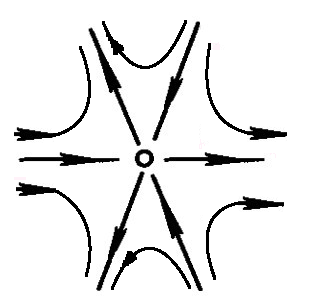}

\caption{\sffamily The dynamics near the origin of $X_H$.}
\end{figure}

\end{subsection}

\begin{subsection}{The specification property}\label{spec}
Let $(H,e, \mathcal{E}_{H,e})$ be a Hamiltonian system and $\Lambda$ be a $X_H^t$-invariant compact subset of $\mathcal{E}_{H,e}$.
A {\em specification} $S=(\tau,P)$ consists of a finite collection
$\tau=\{I_1,\ldots,I_m\}$ of bounded disjoint intervals $I_i=[a_i,b_i]$
of the real line and a map $P:\bigcup_{I_i \in \tau}I_i \rightarrow \Lambda$ such that for any $t_1,t_2 \in I_i$ we have
$$X^{t_2}_H(P(t_1))=X^{t_1}_H(P(t_2)).$$
$S$ is said to be $K$-{\em spaced} if $a_{i+1} \geq b_i+K$ for all $i \in \{1,\cdots,m\}$ and the minimal such
$K$ is called the spacing of $S$. If $\tau=\{I_1,I_2\}$ then $S$ is said to be a {\em weak specification}.
Given $\epsilon>0$, we say that $S$ is $\epsilon$-{\em shadowed} by $x \in \Lambda$ if
$d(X^t_H(x),P(t))<\epsilon$ for all $t \in \bigcup_{I_i \in \tau}I_i$.

We say that $\Lambda$ has the {\em weak specification property} if for any $\epsilon>0$
there exists a $K=K(\epsilon) \in \Rr$ such that any $K$-spaced weak specification $S$ is
$\epsilon$-shadowed by a point of $\Lambda$. In this case the Hamiltonian system $(H,e, \mathcal{E}_{{H,e}{|_{\Lambda}}})$
is said to have the weak specification property.
We say that the Hamiltonian system $(H,e, \mathcal{E}_{H,e})$ has the weak specification property if $\mathcal{E}_{H,e}$
has it.

We say that the Hamiltonian system $(H,e, \mathcal{E}_{H,e})$ has the {\em stable weak specification property} if
there exists a neighbourhood $\mathcal{V}$ of $(H,e, \mathcal{E}_{H,e})$ such that any $(\tilde{H},\tilde{e}, \mathcal{E}_{\tilde{H},\tilde{e}}) \in \mathcal{V}$ has the weak specification property.

\end{subsection}

\begin{subsection}{The transitive and ergodic properties}

A compact energy hypersurface $\mathcal{E}_{H,e}$ is \textit{topologically transitive} if,
for any open and non-empty subsets $U$ and $V$ of $\mathcal{E}_{H,e}$, there is $\tau\in\mathbb{R}$ such that $X_H^{\tau}(U)\cap V\neq\emptyset$.

We say that the Hamiltonian system $(H,e, \mathcal{E}_{H,e})$ is {\em robustly transitive} if
there exists a neighbourhood $\mathcal{V}$ of $(H,e, \mathcal{E}_{H,e})$ such that, for any $(\tilde{H},\tilde{e}, \mathcal{E}_{\tilde{H},\tilde{e}}) \in \mathcal{V}$, the regular energy hypersurface $\mathcal{E}_{\tilde{H},\tilde{e}}$ is transitive.

The probability measure $\mu_{\mathcal{E}_{H,e}}$ is \textit{ergodic} if, for any $X_H^t$-invariant subset $C$ of $\mathcal{E}_{H,e}$, we have that $\mu_{\mathcal{E}_{H,e}}(C)=0$ or $\mu_{\mathcal{E}_{H,e}}(C)=1$.
It is well-known that if a Hamiltonian system $(H,e,\mathcal{E}_{H,e})$ is such that $\mu_{\mathcal{E}_{H,e}}$ is \textit{ergodic} then $\mathcal{E}_{H,e}$ is transitive.

We say that the Hamiltonian system $(H,e, \mathcal{E}_{H,e})$ with $H \in C^3(M,\R)$, is {\em stably ergodic} if
there exists a neighbourhood $\mathcal{V}$ of $(H,e, \mathcal{E}_{H,e})$ such that, for any $(\tilde{H},\tilde{e}, \mathcal{E}_{\tilde{H},\tilde{e}}) \in \mathcal{V}$ with $\tilde{H} \in C^3(M,\R)$, the probability measure $\mu_{\mathcal{E}_{\tilde{H},\tilde{e}}}$ is ergodic.

\end{subsection}

\end{section}

\bigskip

\begin{section}{Precise statement of the results}

A Hamiltonian system $(H,e, \mathcal{E}_{H,e})$ is {\em Anosov} if
$\mathcal{E}_{H,e}$ is uniformly hyperbolic for the Hamiltonian flow
$X_H^t$ associated to $H$.

Our first main result states that the stability of any of the properties: topological stability, shadowing, expansiveness, and specification
guarantees global hyperbolicity:

\begin{maintheorem}\label{mainth1}
Let $(H,e, \mathcal{E}_{H,e})$ be a Hamiltonian system. If any of the following statements hold:
\begin{enumerate}
 \item[\mbox{(1)}] $(H,e, \mathcal{E}_{H,e})$ is robustly topologically stable;
 \item[\mbox{(2)}] $(H,e, \mathcal{E}_{H,e})$ is stably shadowable;
 \item[\mbox{(3)}] $(H,e, \mathcal{E}_{H,e})$ is stably expansive;
 \item[\mbox{(4)}] $(H,e, \mathcal{E}_{H,e})$ has the stable weak specification property,
\end{enumerate}
then $(H,e, \mathcal{E}_{H,e})$ is Anosov.
\end{maintheorem}

It is well-known from classical hyperbolic dynamics that Anosov implies shadowing, expansiveness and topological stability. With respect to weak specification the issue is more subtle. For instance, mixing Anosov  flows satisfy the specification property. However, we consider the following example:

\begin{example}\label{Mix} (Non-mixing Anosov suspension flow) Let be given an Anosov diffeomorphism in a surface $\Sigma$,
 $f\colon \Sigma\rightarrow{\Sigma}$, and a \emph{ceiling function}
$h\colon \Sigma\rightarrow{\mathbb{R}^{+}}$ satisfying $h(x)\geq{\beta}>0$
for all $x\in{\Sigma}$.
We consider the space $M_{h}\subseteq{\Sigma\times{\mathbb{R}^+}}$ defined by
$$
M_h=\{(x,t) \in \Sigma\times{\mathbb{R}^+}: 0 \leq t \leq h(x) \}
$$
with the identification between the pairs $(x,h(x))$ and $(f(x),0)$. The flow defined
on $M_h$ by $S^s(x,r)=(f^{n}(x),r+s-\sum_{i=0}^{n-1}h(f^{i}(x)))$
is an \emph{Anosov suspension flow},
where $n\in{\mathbb{N}_0}$, is uniquely defined by the condition
$$
\sum_{i=0}^{n-1}h(f^{i}(x))\leq{r+s}<\sum_{i=0}^{n}h(f^{i}(x)).
$$
If we choose $f(x)=1$, then the suspension flow cannot be topologically mi\-xing. To see this just observe that the integer iterates of $\Sigma\times(0,1/2)$ are disjoint from $\Sigma\times(1/2,1)$.
\bigskip

\end{example}

A Hamiltonian system $(H,e, \mathcal{E}_{H,e})$ is a {\em Hamiltonian star system} if there exists
a neighbourhood $\mathcal{V}$ of $(H,e, \mathcal{E}_{H,e})$
such that, for any $(\tilde{H},\tilde{e}, \mathcal{E}_{\tilde{H},\tilde{e}})\in\mathcal{V}$, the correspondent regular energy hypersurface $\mathcal{E}_{\tilde{H},\tilde{e}}$ has all the closed orbits hyperbolic.

The next result was proved in~\cite{BFR}, for $n=2$, and recently generalized by the authors in~\cite{BRT}, for $n \geq 2$.

\begin{maintheorem}\label{thm2}
If $(H,e, \mathcal{E}_{H,e})$ is a Hamiltonian star system, then  $(H,e, \mathcal{E}_{H,e})$ is Anosov.
\end{maintheorem}

Thus, Theorem~\ref{mainth1} is a direct consequence of Theorem~\ref{thm2} and Theorem~\ref{thm3}.

\begin{maintheorem}\label{thm3}
Let $(H,e, \mathcal{E}_{H,e})$ be a Hamiltonian system. If any of the following statements hold:
\begin{enumerate}
 \item[\mbox{(a)}] $(H,e, \mathcal{E}_{H,e})$ is robustly topologically stable;
 \item[\mbox{(b)}] $(H,e, \mathcal{E}_{H,e})$ is stably shadowable;
 \item[\mbox{(c)}] $(H,e, \mathcal{E}_{H,e})$ is stably expansive;
 \item[\mbox{(d)}] $(H,e, \mathcal{E}_{H,e})$ has the stable weak specification property,
\end{enumerate}
then $(H,e, \mathcal{E}_{H,e})$ is a Hamiltonian star system.
\end{maintheorem}

It is interesting to note that the specification property implies the topologically mixing property (see Lemma~\ref{mixing}). Moreover, we also note that recently it was proved (see~\cite{BFR2})
that $C^2$-generically Hamiltonian systems are topologically mixing. Clearly, topologically mixing implies transitivity, thus
$C^2$-stable weak specification implies $C^2$-stable transitivity.

Given $1 \leq k \leq n-1$, we recall that a $k$-{\em elliptic closed orbits} has
$2k$ simple non-real eigenvalues of the transversal linear Poincar\'e flow (see Definition \ref{TLPF}) at the period of norm one, and its remaining eigenvalues of norm different from one.
In particular, when $k=n-1$, \emph{(totally) elliptic closed orbits} have all eigenvalues  at the period of norm one, simple and non-real.

Partial hyperbolicity guarantees the decomposition of the normal bundle at energy levels into three invariant subbundles such that, the dynamics is uniformly expanding in one direction, uniformly contracting the other direction and central in the remaining direction (cf. Definition~\ref{PH}). If the central subbundle is trivial the system is Anosov.

A regular energy hypersurface is \emph{far from partially hyperbolic} if it is not in the closure (w.r.t. the $C^2$-topology) of partially hyperbolic surfaces.
Notice that, by structural stability, the union of partially hyperbolic energies is open (\cite{BFR}). Moreover, partially hyperbolic hypersurfaces do not contain elliptic closed orbits. Now, we state the following main result:

\begin{maintheorem}\label{thm4}
Given an open subset $U\subset M$, if a $C^2$-Hamiltonian has a regular energy hypersurface, far from partially hyperbolic, intersecting $U$, then it can be $C^{2}$-approximated by a $C^\infty$-Hamiltonian having a closed elliptic orbit intersecting $U$.
\end{maintheorem}

The previous result generalizes the result stated in \cite[Theorem 6.2]{N} and proved in ~\cite{MBJLD2}. As an almost direct consequence, we arrive at the Newhouse-Arnaud-Saghin-Xia dichotomy for $2n$-dimensional ($n\geq 2$) Hamiltonians (\cite{Ar2002,N,SX}). Recall that, for a $C^2$-generic Hamiltonian, all but finitely many points are regular.

\begin{maintheorem}\label{thm5}
For a $C^2$-generic Hamiltonian $H\in C^2(M,\mathbb{R})$ the union of the partially hyperbolic regular energy hypersurfaces and the closed elliptic orbits, forms a dense subset of $M$.
\end{maintheorem}

At this point it is worth to recall the $4$-dimensional result that motivated the proof of the Hamiltonian version of Franks' lemma (see \S \ref{hamilperturbsection} and Theorem~\ref{Frank}).
Th\'er\`ese Vivier showed in~\cite{V} that any robustly transitive regular energy surface of a $C^2$-Hamiltonian is Anosov.
See also (\cite{HT}) for the symplectomorphisms case. It is easy to see that our results imply the multidimensional version of Vivier's theorem.
In fact, if a regular energy hypersurface $\mathcal{E}_{H,e}$ of a $C^2$-Hamiltonian $H$ is far from partial hyperbolicity, then, by Theorem~\ref{thm4}, there exists a $C^2$-close $C^\infty$-Hamiltonian with an elliptic closed orbit on a nearby regular energy hypersurface.
This invalidates the chance of robust transitivity for $H$ according to a KAM-type criterium (see~\cite[Corollary 9]{V}).
The same argument shows that the presence of a regular energy hypersurface $\mathcal{E}_{H,e}$ of a $C^2$-Hamiltonian $H$ which is far from partial hyperbolicity invalidates the chance of stable ergodicity.

\begin{maincorollary}\label{cor4a}
Let $(H,e, \mathcal{E}_{H,e})$ be a robustly transitive Hamiltonian system.
Then $\mathcal{E}_{H,e}$ is partially hyperbolic.
\end{maincorollary}

\begin{maincorollary}\label{cor4b}
Let $(H,e, \mathcal{E}_{H,e})$ be a stably ergodic Hamiltonian system.
Then $\mathcal{E}_{H,e}$ is partially hyperbolic.
\end{maincorollary}

Finally, we obtain the following result that states that any robustly weakly shadowable regular energy hypersurface of a $C^2$-Ha\-mil\-to\-nian is partially hyperbolic.

\begin{maintheorem}\label{thm6}
Let $(H,e, \mathcal{E}_{H,e})$ be a stably weakly shadowable Hamiltonian system.
Then $\mathcal{E}_{H,e}$ is partially hyperbolic.
\end{maintheorem}

\end{section}

\begin{section}{Perturbation lemmas}\label{hamilperturbsection}

In this section we present three key perturbation results for Hamiltonians that we shall use in the sequel.
The first one (Theorem~\ref{Pasting}) is a version of the $C^1$-pasting lemma (see~\cite[Theorem 3.1]{AM}) for Hamiltonians.
Actually, in the Hamiltonian setting, the proof of this result is much more simple (see ~\cite{BFR2}).
The second perturbation result (Theorem~\ref{Frank}), due to Vivier, is a version of Franks's lemma for Hamiltonians
(see~\cite[Theorem 1]{V}).
Roughly speaking, it says that we can realize a Hamiltonian corresponding to a given perturbation of the transversal linear Poincar\'{e} flow. The last perturbation result (Theorem~\ref{Suspension}) is a Hamiltonian suspension theorem
(see~\cite[Theorem 3]{MBJLD4}), specially useful for the conversion of perturbative results between symplectomorphisms
and Hamiltonian flows. Indeed, if we perturb the Poincar\'{e} map of a periodic orbit, there is a nearby Hamiltonian
realizing the new map.

\begin{theorem}\label{Pasting} {(Pasting lemma for Hamiltonians)}  Fix $H \in C^r(M,\mathbb{R})$, $2 \leq r \leq \infty$, and let $K$ be a compact subset of $M$ and $U$ a small neighbourhood of $K$.
Given $\epsilon>0$, there exists $\delta >0$ such that if $H_1 \in C^l(M,\mathbb{R})$, for $2 \leq l \leq \infty$, is $\delta$-$C^{\min\{r,l\} }$-close to $H$ on $U$
then there exist $H_0 \in C^l(M,\mathbb{R})$ and a closed set $V$ such that:
\begin{itemize}
        \item $K \subset V \subset U$;
	\item $H_0=H_1$ on $V$;
	\item $H_0=H$ on $U^c$;
	\item $H_0$ is $\epsilon$-$C^{\min\{r,l\} }$-close to $H$.
\end{itemize}
\end{theorem}

Let $x\in M$ be a regular point of a Hamiltonian $H$ and define the \textsl{arc} $X_H^{[t_1,t_2]}(x)=\{X_H^t(x), \, t \in [t_1,t_2]\}.$
Given a transversal section $\bf{\Sigma}$ to the flow at $x$, a \textsl{flowbox} associated to $\bf{\Sigma}$ is defined by $\mathcal{F}(x)=X_H^{[-\tau_1,\tau_2]}(\bf{\Sigma}),$ where $\tau_1,\tau_2$ are chosen small such that $\mathcal{F}(x)$ is a neighborhood of $x$ foliated by regular orbits.

\begin{theorem}\label{Frank}(Franks' lemma for Hamiltonians) \label{Franks}
Take $H \in C^r(M,\mathbb{R})$, $2 \leq r \leq \infty$, $\epsilon>0$, $\tau >0$ and $x\in M$.
Then, there exists $\delta>0$ such that for any flowbox $\mathcal{F}(x)$ of an injective arc of orbit $X_{H}^{[0,t]}(x)$, with $t\geq \tau$, and a transversal symplectic $\delta$-perturbation $\Psi$ of $\Phi_{{H}}^t(x)$, there is $H_0\in C^\ell(M,\mathbb{R})$ with $\ell = \max\{2,k-1\}$ satisfying:
\begin{itemize}
	\item $H_0$ is $\epsilon$-$C^2$-close to $H$;
	\item $\Phi_{{H_0}}^t(x)=\Psi$;
	\item $H=H_0$ on $X_{H}^{[0,t]}(x)\cup (M\backslash \mathcal{F}(x))$.
\end{itemize}
\end{theorem}

Consider a Hamiltonian system $(H,e, \mathcal{E}_{H,e})$ and a periodic point $p \in \mathcal{E}_{H,e}$ with period $\pi$.
At $p$ consider a transversal ${\bf{\Sigma}} \subset M$ to the flow, i.e. a local $(2n-1)$-submanifold
for which $X_H$ is nowhere tangencial. Define the $2n-2$ symplectic submanifold
$$\Sigma={\bf{\Sigma}} \cap  \mathcal{E}_{H,e}.$$
Thus, for any $x \in \Sigma$
$$T_x \mathcal{E}_{H,e}  =T_x \Sigma \oplus \mathbb{R} X_H(x).$$
Let $U \subset M$ be some open neighbourhood of $p$ and $V=U \cap \Sigma.$ The {\em Poincar\'{e} (section) map}
$f: V \rightarrow \Sigma$ is the return map of $X_H^t$ to $\Sigma$. It is given by
$$f(x)=X_H^{\tau(x)}(x), \, x \in V,$$
where $\tau$ is the return time to $\Sigma$ defined implicitely by the relation
$X_H^{\tau(x)}(x) \in \Sigma$ and satisfying $\tau(p)=\pi$. In addition, $p$ is a fixed point of $f$.
Notice that one needs to assume that $U$ is a small neighbourhood of $p$.
Thus, $f$ is a $C^1$-symplectomorphism between $V$ and its image. Moreover, any two Poincar\'{e} section maps
of the same closed orbit are conjugate by a symplectomorphism.

\begin{theorem}\label{Suspension} {(Hamiltonian suspension)} Let $H \in C^\infty(M,\mathbb{R})$ with
 Poincar\'{e} map $f$ at a periodic point $p$. Then, for any $\epsilon>0$ there is $\delta>0$ such that
for any symplectomorphism $f_0$ being $\delta$-$C^3$-close to $f$, there is a Hamiltonian $H_0$ $\epsilon$-$C^2$-close
with Poincar\'{e} map $f_0$.
\end{theorem}

\bigskip

\begin{section}{Hyperbolicity \emph{versus} stable shades (proof of Theorem~\ref{thm3})}

We shall start by a key lemma that states that the presence of a non-hyperbolic periodic point $p$ for a Hamiltonian $H$
ensures the existence of an Hamiltonian $H_1$, arbitrarily $C^2$-close to $H$, exhibiting a
continuum of periodic points close to $p$.

Consider a Hamiltonian system $(H,e, \mathcal{E}_{H,e})$  and a periodic point $p \in \mathcal{E}_{H,e}$
with period $\pi$. Let
$\Sigma_p^c$ denote a submanifold of $\Sigma$, associated to $p$, such that $T_p \Sigma_p^c \oplus \mathbb{R}X_H(x)=\mathcal{N}^c_p \oplus \mathbb{R} X_H(x)$, where $\mathcal{N}^c_p$ denotes the subspace of $\mathcal{N}_p$ associated with norm-one eigenvalues of $\Phi_H^{\pi}(p)$.

\begin{lemma}\label{continuum}(Explosion of periodic orbits)
Let $(H,e, \mathcal{E}_{H,e})$ be a Hamiltonian system and let $p \in \mathcal{E}_{H,e}$ be a non-hyperbolic periodic point.
Then, there exists a Hamiltonian system $(H_1,e_1, \mathcal{E}_{H_1,e_1})$, arbitrarily close to
$(H,e, \mathcal{E}_{H,e})$, such that
$H_1$ has a non-hyperbolic periodic point $q \in \mathcal{E}_{H_1,e_1}$, close to $p$, and such that every point in a small neighborhood of $q$,  in $\Sigma^c_q$, is a periodic point of $H_1$.
\end{lemma}

\begin{proof}
Let $(H,e, \mathcal{E}_{H,e})$ be a Hamiltonian system and let $p \in \mathcal{E}_{H,e}$ be a non-hyperbolic periodic point
of period $\pi$.
If $H \in C^{\infty}(M,\R)$, take $H_0=H$, otherwise we use the Pasting lemma for Hamiltonians (Theorem~\ref{Pasting})
to obtain a Hamiltonian
$\overline{H} \in C^{\infty}(M,\R)$
such that $\overline{H}$ is arbitrarily $C^2$-close to $H$ and such that $\overline{H}$
has a periodic point $\overline{p}$, close to $p$, with period $\overline{\pi}$ close to $\pi$.
We observe that $\overline{p}$ may not be the analytic continuation of $p$ and
this is precisely the case when $1$ is an eigenvalue of $\Phi^{\pi}_{H}(p)$.
If $\overline{p}$ is not hyperbolic, take $H_0=\overline{H}$.
If $\overline{p}$ is hyperbolic
then, since $\overline{H}$ is arbitrarily $C^2$-close
to $H$, the distance between the spectrum of
$\Phi^{\overline{\pi}}_{\overline{H}}(\overline{p})$ and $\mathbb{S}^1$ can be taken
arbitrarily close to zero (weak hyperbolicity). Now, we are in position to apply Frank's lemma for Hamiltonians (Theorem~\ref{Frank})
to obtain a new Hamiltonian $H_0 \in C^{\infty}(M,\R)$, $C^2$-close
to $\overline{H}$ and having a non-hyperbolic periodic point $q$ close to $\overline{p}$.

Clearly, the Poincar\'{e} map $f_0$
at $q$, associated to the Hamiltonian flow $X^t_{H_0}$, is a $C^{\infty}$ local symplectomorphism.
In order to go on with the argument and obtain the Hamiltonian $H_1$, the first step is to use the Weak pasting lemma for symplectomorphisms~\cite[Lemma 3.9]{AM} to change the Poincar\'{e} map $f_0$ by its derivative. In this way we get a symplectomorphism $f_1$, arbitrarily $C^{\infty}$-close to $f_0$, such that (in the local canonical coordinates mentioned in \cite{MZ} and given by Darboux's theorem)
$f_1$ is linear and equal to $Df_0$ in a neighbourhood of the periodic non-hyperbolic periodic point $q$.

Next, we use the Hamiltonian suspension theorem (Theorem~\ref{Suspension}) to realize $f_1$, i.e.,
in order to obtain a Hamiltonian $H_1 \in C^\infty(M,\R)$
such that $f_1$ is linear and equal to $Df_0$ in a neighbourhood of the non-hyperbolic periodic point $q$.

Moreover, the existence of an eigenvalue, $\lambda$, with modulus equal to one is associated to a symplectic invariant
two-dimensional subspace $E$ contained in the subspace $\mathcal{N}^c_q$, associated to norm-one eigenvalues.
Furthermore, up to a perturbation using again the Hamiltonian suspension theorem, $\lambda$ can be taken rational, thus creating periodic points
related with $E$.
This argument can be repeated for each norm-one eigenvalue, if necessary (see the proof of~(\cite[Theorem 2]{BRT}) for the details).
This ensures the existence of a Hamiltonian system $(H_1,e_1, \mathcal{E}_{H_1,e_1})$ arbitrarily close to
$(H,e, \mathcal{E}_{H,e})$ such that $q \in \mathcal{E}_{H_1,e_1}$, and of a small neighborhood of $q$
in $\Sigma_q^c$ composed of periodic points.
\end{proof}

\begin{remark}\label{linear}
 Let $(H_1,e_1, \mathcal{E}_{H_1,e_1})$ and $q \in \mathcal{E}_{H_1,e_1}$ be given by Lemma~\ref{continuum}. We observe that the proof of Lemma~\ref{continuum}
guarantees that  $(H_1,e_1, \mathcal{E}_{H_1,e_1})$ is such that the Poincar\'{e} map $f_1$ at $q$ associated to $X^t_{H_1}$ is a linear map (in the local canonical coordinates mentioned above). This fact will be implicitly used in the proofs of \mbox{(a)}, \mbox{(b)} and \mbox{(c)} of Theorem~\ref{thm3}
and in the proof of Proposition~\ref{ML}.
\end{remark}

Now we are in position to prove item \mbox{(a)} of Theorem~\ref{thm3}.

\begin{proof}
Given a robustly topologically stable Hamiltonian system $(H,e, \mathcal{E}_{H,e})$, we prove that all its periodic orbits in
$\mathcal{E}_{H,e}$ are hyperbolic; from
this it follows that $(H,e, \mathcal{E}_{H,e})$ is a Hamiltonian star system.

By contradiction, let us assume that there exists a robustly topologically stable Hamiltonian system $(H,e, \mathcal{E}_{H,e})$
having a non-hyperbolic periodic point $p \in \mathcal{E}_{H,e}$. If follows from Lemma~\ref{continuum} that there exists a robustly topologically stable Hamiltonian system
$(H_1,e_1, \mathcal{E}_{H_1,e_1})$, arbitrarily close to $(H,e, \mathcal{E}_{H,e})$,
and there exists a non-hyperbolic periodic point $q \in \mathcal{E}_{H_1,e_1}$ of $H_1$ for which every point in a  small neighborhood of $q$,
in $\Sigma_q^c$,
is a periodic point of $H_1$.

Finally, we
approximate $(H_1,e_1, \mathcal{E}_{H_1,e_1})$ by $(H_2,e_2, \mathcal{E}_{H_2,e_2})$, also robustly topologically stable, such that $q$ is an hyperbolic periodic point or an isolated $k$-elliptic periodic point (for $H_2$).
This is a contradiction because $(H_2,e_2, \mathcal{E}_{H_2,e_2})$ is semiconjugated to $(H_1,e_1, \mathcal{E}_{H_1,e_1})$, although there is an $H_1$-orbit (different from $q$)
contained in a small neighbourhood of $q$ and the same cannot occur for $H_2$ because $q$ is a hyperbolic
periodic point or an isolated $k$-elliptic periodic point of $H_2$.
\end{proof}

\bigskip

Let us now prove item \mbox{(b)} of Theorem~\ref{thm3}.

\begin{proof}
Given a stably shadowable Hamiltonian system $(H,e, \mathcal{E}_{H,e})$, we prove that all its periodic orbits in
$\mathcal{E}_{H,e}$ are hyperbolic; from
this it follows that $(H,e, \mathcal{E}_{H,e})$ is a Hamiltonian star system.

By contradiction, let us assume that there exists a stably shadowable Hamiltonian system $(H,e, \mathcal{E}_{H,e})$ having a non-hyperbolic periodic point $p \in \mathcal{E}_{H,e}$. It follows from
Lemma~\ref{continuum} that there exists a stably shadowable Hamiltonian system $(H_1,e_1, \mathcal{E}_{H_1,e_1})$, arbitrarily close to $(H,e, \mathcal{E}_{H,e})$,
and there exists a non-hyperbolic periodic point $q \in \mathcal{E}_{H_1,e_1}$ of $H_1$ for which every point, say in a $\xi$-neighborhood of $q$,
in $\Sigma_q^c$,
is a periodic point of $H_1$.

Since $(H_1,e_1, \mathcal{E}_{H_1,e_1})$ has the shadowing property fixing $\epsilon \in (0,\frac{\xi}{4})$, there exist
$\delta \in (0,\epsilon)$ and $T>0$ such that every $(\delta,T)$-pseudo-orbit $((x_i),(t_i))_{i \in \Z}$ is
$\epsilon$-shadowed by some orbit of $H_1$.

Let $x_0=q$.
Take $y\in \Sigma_q^c$ such that $d(q,y)=\frac{3\xi}{4}>2\epsilon$ and fix $\delta \in (0,\epsilon)$ sufficiently
small. We construct a bi-infinite sequence of points $((x_i), (t_i))_{i \in \Z}$ with $x_i \in \Sigma_q^c$
such that $((x_i), (t_i))_{i \in \Z}$ is a $(\delta,T)$-pseudo orbit
for some $T>0$.
For fixed $k \in \Z$, let $x_k=y$. There exist $x_i \in \Sigma_q^c$, $i \in \Z$, such that:
\begin{enumerate}
 \item[$\bullet$] $x_i=x_0$ and $t_i=\pi$ for $i \leq 0$;
\item[$\bullet$] $d(x_{i},x_{i-1}) < \delta$ and $t_i=\pi$ for $1 \leq i \leq k$;
\item[$\bullet$] $x_i=x_k$ and $t_i=\pi$ for $i >k$.
\end{enumerate}

Observe that we are considering that the return time at the transversal section is the same and equal to $\pi$. Clearly, it is not exactly equal to $\pi$, however
it is as close to $\pi$ as we want by just decreasing the $\xi$-neighborhood.
Therefore, $((x_i), (\pi))_{i \in \Z}$ is a $(\delta,T)$-pseudo-orbit for some $T>0$ such that $\pi \geq T$.

By the shadowing property, there is a point $z \in \mathcal{E}_{H,e}$ and a reparametrization $\alpha \in \mbox{Rep}(\epsilon)$
such that
$d(X_{H_1}^{\alpha(t)}(z),x_0 \star t)<\epsilon$, for every $t \in \R$.
Hence, cannot have forward/backward expansion and so,
$z \in \Sigma_q^c$. However, since $(H_1,e_1, \mathcal{E}_{H_1,e_1})$ has the shadowing property and $x_0 \star k \pi=x_k$, we have that
$$
d(q,y) \leq d(q,X_{H_1}^{\alpha(k \pi)}(z))+d(X_{H_1}^{\alpha(k \pi)}(z),x_k)<2 \epsilon,
$$
which is a contradiction.
\end{proof}

Now we prove item \mbox{(c)} of Theorem~\ref{thm3}.

\begin{proof}
Given a stably expansive Hamiltonian system $(H,e, \mathcal{E}_{H,e})$, we prove that all its periodic orbits are hyperbolic; from
this it follows that $(H,e, \mathcal{E}_{H,e})$ is a Hamiltonian star system.

By contradiction, let us assume that there exists a stably expansive Hamiltonian system $(H,e, \mathcal{E}_{H,e})$ having a non-hyperbolic periodic point $p \in \mathcal{E}_{H,e}$. If follows from Lemma~\ref{continuum} that there exists a stably expansive Hamiltonian system $(H_1,e_1, \mathcal{E}_{H_1,e_1})$, arbitrarily close to $(H,e, \mathcal{E}_{H,e})$,
and there exists a non-hyperbolic periodic point $q \in \mathcal{E}_{H_1,e_1}$ of period $\pi$ of $H_1$ for which every point in a small neighborhood of $q$,
in $\Sigma_q^c$,
is a periodic point of $H_1$ with period close to $\pi$.

Finally, we just have to pick two points $x,y \in \Sigma_q^c$ sufficiently close in order to obtain
$d(X^t_{H_1}(x),X^{t}_{H_1}(y))<\delta$ for all $t \in \Rr$. It is clear that
$(H_1,e_1, \mathcal{E}_{H_1,e_1})$ can not be expansive which is a contradiction and Theorem~\ref{thm3}\,\mbox{(c)} is proved.
\end{proof}

To prove Theorem~\ref{thm3}\,\mbox{(d)}, we shall start by deducing some consequences of the weak specification property. Let us first recall that
a compact energy hypersurface $\mathcal{E}_{H,e}$ of a Hamiltonian $(H,e, \mathcal{E}_{H,e})$ is {\em topologically mixing} if, for any open and non-empty
subsets of $\mathcal{E}_{H,e}$, say $U$ and $V$, there is $\tau \in \Rr$ such that
$X_H^t(U) \cap V \neq \emptyset$, for any $t \geq \tau$. The first lemma is a particular case of~\cite[Lemma 3.1]{ASS}.

\begin{lemma}\label{mixing} If a Hamiltonian system $(H,e, \mathcal{E}_{H,e})$ has the weak specification property, then
$\mathcal{E}_{H,e}$ is topologically mixing.
\end{lemma}

Let $(H,e, \mathcal{E}_{H,e})$ be a Hamiltonian system and let
$p \in \mathcal{E}_{H,e}$ be a periodic point of period $\pi$ such that the spectrum of $\Phi_H^{\pi}(p)$ outside the unit circle is a non-empty set. Let $\Ss^0(H,p)$ stand for the spectrum outside the unit circle. Observe that this set contains both eigenvalues with modulus greater than one and  smaller than one.

We define the strong stable and stable manifolds of $p$ as:
$$\displaystyle{W^{ss}(p)=\{y \in \mathcal{E}_{H,e}: \lim_{t \rightarrow +\infty} d(X_H^t(y),X_H^t(p))=0\}}$$
and
$$\displaystyle{W^{s}(\mathcal{O}(p))=\bigcup_{t \in \Rr} W^{ss}(X^t_H(p))},$$
where $\mathcal{O}(x)$ stands for the orbit of $x$.
For small $\epsilon>0$, the local strong stable manifold is defined as
$$\displaystyle{W^{ss}_{\epsilon}(p)=\{y \in \mathcal{E}_{H,e};  d(X_H^t(y),X_H^t(p))<\epsilon \,\, \text{if} \,\, t \geq 0\}}.$$
By the stable manifold theorem, there exists an $\epsilon=\epsilon(p)>0$ such that
$$\displaystyle{W^{ss}(p)=\bigcup_{t \geq 0} X_H^{-t}(W^{ss}_{\epsilon}(X^t_H(p)))}.$$
Analogous definitions hold for unstable manifolds.

Next result is an adaptation of~\cite[Theorem 3.3]{ASS}.

\begin{lemma}\label{wuwsnotempty} If a Hamiltonian system $(H,e, \mathcal{E}_{H,e})$ has the weak specification property, then
for every distinct periodic points $p, q \in \mathcal{E}_{H,e}$
such that $\Ss^0(H,p) \neq \emptyset$ and $\Ss^0(H,q) \neq \emptyset$, we have that
$W^u(\mathcal{O}(p)) \cap W^s(\mathcal{O}(q)) \neq \emptyset.$
\end{lemma}

\begin{proof}
We denote by $\epsilon(p)$ the size of the local strong unstable manifold of $p$ and by
$\epsilon(q)$ the size of the local strong stable manifold of $q$.
Let $\epsilon=\min\{\epsilon(p),\epsilon(q)\}$, and let $K=K(\epsilon)$ be given by the weak specification property.
If $t>0$ then take $I_1=[0,t]$ and $I_2=[K+t,K+2t]$. Now define
$P(s)=X^{s-t}_H(p)$ if $s \in I_1$ and $P(s)=X^{s-K-t}_H(q)$ if $s \in I_2$. Note that this is a
$K$-spaced weak specification.

So, there exists $x_t$ which shadows this weak specification:
$$d(X_H^{s}(x_t),P(s)) \leq \epsilon \,\, \text{if} \,\, s \in I_1 \cup I_2.$$
Using the change of variable $u=t-s$, for every $u \in [0,t]$ we have:
$$d(X_H^{-u}(X_H^t(x_t)),X_H^{-u}(p))= d(X_H^{t-u}(x_t),X_H^{-u}(p)) \leq \epsilon$$
and using $u=s-K-t$, for every $u \in [0,t]$ we have
$$d(X_H^{u}(X_H^{K+t}(x_t)),X_H^{u}(q)) \leq \epsilon.$$
If $y_t=X_H^t(x_t)$ then we can assume that $y_t \rightarrow y$. And taking limits in the previous inequalities we obtain
$$d(X_H^{-u}(y),X_H^{-u}(p)) \leq \epsilon \,\, \text{for every} \,\, u \geq 0, \,\, \text{and}$$
$$d(X_H^{u}(X_H^K(y)),X_H^{u}(q)) \leq \epsilon \,\, \text{for every} \,\, u \geq 0.$$
The first one says that $y \in W^{uu}_{\epsilon}(p) \subset W^{u}(\mathcal{O}(p))$ and the second one says that
$X^K_H(y) \in W^{ss}_{\epsilon}(q)$, hence $y \in W^{s}(\mathcal{O}(p)).$
\end{proof}

\begin{proposition}\label{sph}
If $(H,e, \mathcal{E}_{H,e})$ satisfies the stable weak specification property, then $\mathcal{E}_{H,e}$ is partially hyperbolic.
In particular, due to Remark~\ref{Mane}, if $n=2$, then $\mathcal{E}_{H,e}$ is hyperbolic.
\end{proposition}

\begin{proof}
The proof is by contradiction; let us assume there exists a Hamiltonian system $(H,e, \mathcal{E}_{H,e})$
that has the
stable weak specification property and such that
$\mathcal{E}_{H,e}$ is not partially hyperbolic.
Then, by Theorem~\ref{thm4}, there exists a $C^2$-close $C^\infty$-Hamiltonian $H_0$ with an elliptic closed orbit
on a nearby regular energy hypersurface $\mathcal{E}_{{H_0},{e_0}}$.
This invalidates  the chance of mixing for  $\mathcal{E}_{{H_0},{e_0}}$ according to a KAM-type criterium (see~\cite[Corollary 9]{V}),
which contradicts Lemma~\ref{mixing}.
\end{proof}

We recall that a Hamiltonian system $(H,e, \mathcal{E}_{H,e})$ is a \textit{Kupka-Smale Hamiltonian system} if (see~\cite{R0}):
\begin{enumerate}
 \item the union of the hyperbolic and $k$-elliptic closed orbits ($1 \leq k \leq n-1$) in $\mathcal{E}_{H,e}$ is dense in $\mathcal{E}_{H,e}$;
\item  the intersection of $W^s(\mathcal{O}(p))$ with $W^u(\mathcal{O}(q))$ is transversal, for any closed orbits $\mathcal{O}(p)$ and
$\mathcal{O}(q)$.
\end{enumerate}

\begin{lemma}\label{hypnonhyp} Let $(H,e, \mathcal{E}_{H,e})$ be a Hamiltonian system satisfying the stable weak specification property and let $\mathcal{V}$ be a neighbourhood  of $(H,e, \mathcal{E}_{H,e})$ such that any $({H_0},{e_0}, \mathcal{E}_{{H_0},{e_0}}) \in \mathcal{V}$ has the weak specification property. Then, every Kupka-Smale
Hamiltonian system in $\mathcal{V}$ has all periodic points of hyperbolic type.
\end{lemma}

 \begin{proof}
Let $\mathcal{V}$ be a neighbourhood of
$(H,e, \mathcal{E}_{H,e})$ as in the hypothesis of the lemma.
Let $p, q \in \mathcal{E}_{{H_0},{e_0}}$ be two periodic points of a Kupka-Smale Hamiltonian system $({H_0},{e_0}, \mathcal{E}_{{H_0},{e_0}}) \in \mathcal{V}$ and suppose, by contradiction, that $p$ is a non-hyperbolic periodic point. Then, $\dim W^u(\mathcal{O}(p)) < (2n-2)/2$ and, therefore, $\dim W^u(\mathcal{O}(p)) + \dim W^s(\mathcal{O}(q)) < 2n-2$. Since, the stable/unstable manifolds intersect in a tranversal way, we must have
$W^u(\mathcal{O}(p)) \cap W^s(\mathcal{O}(q)) = \emptyset.$ But this contradicts Lemma~\ref{wuwsnotempty}.
 \end{proof}

\begin{remark}\label{elliptic open}
Fix some Hamiltonian system $(H,e, \mathcal{E}_{H,e})$ such that $\mathcal{E}_{H,e}$ has a $k$-elliptic closed orbit, $1 \leq k \leq n-1$. We get that the analytic continuation of $\mathcal{E}_{H,e}$, $\mathcal{E}_{\tilde{H},\tilde{e}}$, has still a
$k$-elliptic closed orbit (its analytic continuation). Therefore, the set of Hamiltonians exhibiting $k$-elliptic
($1 \leq k \leq n-1$) closed orbits is open in $C^2(M,\mathbb{R})$ (see e.g. \cite{R0}).
\end{remark}

\begin{lemma}\label{keySP}
Let $(H,e, \mathcal{E}_{H,e})$ be a Hamiltonian system and let $p \in \mathcal{E}_{H,e}$ be a non-hyperbolic periodic point.
Then, there exists  $({H_0},{e_0}, \mathcal{E}_{{H_0},{e_0}})$, arbitrarily close to $(H,e, \mathcal{E}_{H,e})$,
such that $({H_0},{e_0}, \mathcal{E}_{{H_0},{e_0}})$ is a Kupka-Smale Hamiltonion system exhibiting $1$-elliptic periodic points.
\end{lemma}

\begin{proof}
Let $(H,e, \mathcal{E}_{H,e})$ be a Hamiltonian system with a non-hyperbolic periodic point $p \in \mathcal{E}_{H,e}$. As the boundary of the Anosov Hamiltonian systems has no isolated points (see~\cite[Corollary 1]{BRT}), it follows from Newhouse dicothomy for Hamiltonians~\cite{N, MBJLD2} that $H$ can be $C^2$-approximated by a Hamiltonian exhibiting $1$-elliptic periodic points.
Since, by Remark~\ref{elliptic open},  $1$-elliptic periodic points are stable, it follows from Robinson's version of the Kupka-Smale theorem (see~\cite{R0}) that there exists a Kupka-Smale Hamiltonian system
$(H_0,e_0, \mathcal{E}_{H_0,e_0})$, arbitrarily close to $(H,e, \mathcal{E}_{H,e})$, such that
$H_0$ has a $1$-elliptic periodic point in $\mathcal{E}_{H_0,e_0}$.
\end{proof}

Now we are in position to prove item \mbox{(d)} of Theorem~\ref{thm3}.

\begin{proof}
Given a Hamiltonian system $(H,e, \mathcal{E}_{H,e})$ satisfying the stable weak specification property, we prove that all its periodic orbits are hyperbolic; from
this it follows that $(H,e, \mathcal{E}_{H,e})$ is a Hamiltonian star system.

By contradiction, let us assume that there exists a Hamiltonian system $(H,e, \mathcal{E}_{H,e})$
satisfying the stable weak specification property and
having a non-hyperbolic periodic point $p \in \mathcal{E}_{H,e}$. Let $\mathcal{V}$ be a neighbourhood of
$(H,e, \mathcal{E}_{H,e})$ such that the weak specification property is verified.
Using Lemma~\ref{keySP} there exists a Kupka-Smale Hamiltonian $({H_0},{e_0}, \mathcal{E}_{{H_0},{e_0}}) \in \mathcal{V}$
such that ${H_0}$ has a non-hyperbolic periodic point which contradicts
Lemma~\ref{hypnonhyp}.

\end{proof}

\end{section}

%%%%%%%%%%%%%%%%%%%%%%%%%%%%%%%%%%%%%%%%%%%%%%%%%%%%%%%%%%%%%%%%%%%%%%%%%%%%%
\section{Partial hyperbolicity \emph{versus} dense elliptic orbits (proof of Theorems~\ref{thm4} and~\ref{thm5})}\label{arnaud}
%%%%%%%%%%%%%%%%%%%%%%%%%%%%%%%%%%%%%%%%%%%%%%%%%%%%%%%%%%%%%%%%%%%%%%%%%%%%%
\subsection{Proof of Theorem~\ref{thm5}}

A Hamiltonian system $(H,e, \mathcal{E}_{H,e})$ is \textsl{partially hyperbolic} if $\mathcal{E}_{H,e}$ is partially hyperbolic.
Let $\mathcal{PH}^2_{\omega}(M)\subset C^2(M,\mathbb{R})$ denote the subset of partially hyperbolic Hamiltonians \footnote{Observe that, due to Remark~\ref{Mane}, if $n=2$, then $\mathcal{PH}^2_{\omega}(M)$ is equal to the Anosov Hamiltonian systems.}.

Fix $H\in \mathcal{PH}^2_{\omega}(M)$ and let  $e\in H(M)$ be an energy such that the subset $H^{-1}(\{e\})$ has a partial hyperbolic component $\mathcal{E}_{H,e}$. For any $\tilde{H}$ arbitrarilly $C^2$-close to $H$ and $\tilde{e}$ arbitrarially close to $e$, we get that the analytic continuation of $\mathcal{E}_{H,e}$, $\mathcal{E}_{\tilde{H},\tilde{e}}$, is still partially hyperbolic. Thus, in other words, partial hyperbolicity is an \emph{open} property. The proof is similar to the openness of the hyperbolicity done in \cite{BFR} and mainly uses cone field arguments.

The proof of Theorem~\ref{thm5} is a consequence of Theorem \ref{thm4}.

\begin{proof}
Consider the set
$$
\mathcal{G}=C^2(M,\mathbb{R}) \times M
$$
endowed with the
product topology associated to the $C^2$-topology in $C^2(M,\mathbb{R})$ and with the topology inherited by the Riemannian structure in $M$.
Given $p\in M$, let $\mathcal{E}_{H,e}$ be the energy surface passing through $p$.
As we mention before the subset
$$
\mathcal{PH}:=\left\{(H,p)\in\mathcal{G}\colon \mathcal{E}_{H,e} \text{ is a partially hyperbolic regular energy hypersurface} \right\}
$$
is open.
Let $\overline{\mathcal{PH}}$ be its closure (w.r.t. the $C^2$-topology) with complement $\mathcal{N}=\mathcal{G}\setminus \overline{\mathcal{PH}}$.

Given $\epsilon>0$ and an open set $\mathcal{U}\subset \mathcal{N}$, define the subset $\mathcal{O}(\mathcal{U},\epsilon)$ of pairs $(H,p)\in\mathcal{U}$ for which $H$ has a closed elliptic orbit intersecting the $(2n-1)$-dim ball $B(p,\epsilon)\cap \mathcal{E}_{H,e}$. This is possible due to Theorem~\ref{thm4}.
It follows from Theorem~\ref{thm4} and the fact that (totally)-elliptic orbits are stable (Remark~\ref{elliptic open}), that $\mathcal{O}(\mathcal{U},\epsilon)$ is dense and open in $\mathcal{U}$.

Let $(\epsilon_k)_{k\in\mathbb{N}_0}$ be a positive sequence such that $\epsilon_k {\rightarrow} 0$ when ${k\rightarrow 0}$.
Then, define recursively the sequence of dense and open sets $\mathcal{U}_0 = \mathcal{N}$ and $
\mathcal{U}_{k}=\mathcal{O}(\mathcal{U}_{k-1}, \epsilon_{k-1})$, $ k\in\mathbb{N}$.
Notice that $\bigcap_{k\in\mathbb{N}}\mathcal{U}_k$ is the set of pairs $(H,p)$ yielding the property that $p$ is accumulated by closed elliptic orbits for $H$.

Finally, the above implies that, for each $k\in\mathbb{N}$,  $\mathcal{PH}\cup\mathcal{U}_k$ is open and dense in $\mathcal{G}$, and
$$
\mathfrak{F}:=\bigcap_{k\in\mathbb{N}}(\mathcal{PH}\cup\mathcal{U}_k)=\mathcal{PH}\cup \bigcap_{k\in\mathbb{N}}\mathcal{U}_k
$$
is residual.
By~\cite[Proposition A.7]{BF}, we write
$$
\mathfrak{F}=\bigcup_{H\in\mathfrak{R}} \{H\} \times\mathfrak{M}_H,
$$
where $\mathfrak{R}$ is $C^2$-residual in $C^2(M,\mathbb{R})$ and, for each $H\in\mathfrak{R}$, $\mathfrak{M}_H$ is a residual subset of $M$, having the following property:
if $H\in\mathfrak{R}$ and $p\in\mathfrak{M}_H$, then $\mathcal{E}_{H,e}$ is partially hyperbolic or $p$ is accumulated by closed elliptic orbits.

\end{proof}

\subsection{Proof of Theorem~\ref{thm4}}

We begin by considering the following result which is a kind of closing lemma of strong type.

\begin{lemma}\label{SXhamilt}
For any homoclinic point $z$ associated to the periodic hyperbolic point $x$ of $H\in C^\infty(M,\mathbb{R})$, there exists an arbitrarily small $C^2$-perturbation of $H$ supported in a small neighborhood of $x$ such that $z$ becomes a periodic point.
\end{lemma}

\begin{proof}
By \cite[Lemma 10]{SX}, for any homoclinic point $z$ associated to the periodic hyperbolic point $x$ of $f\in \text{Diff}^3_{\omega}(M^{2n-2})$, there exists an arbitrarily small $C^3$ perturbation of $f$,  $\tilde{f}\in\text{Diff}^3_{\omega}(M^{2n-2})$, supported in a small neighborhood of $x$ such that $z$ becomes a periodic point.

Since periodic points are dense in the homoclinic class, we can choose a periodic point $p$ close to $x$. We consider the Poincar\'e map of $H$ in a small transversal section at $x$ and define it as the symplectic map $f$ obtained in \cite[Lemma 10]{SX}. Finally, the Hamiltonians suspension theorem (Theorem~\ref{Suspension}), gives the perturbation required in the statement of the lemma.

\end{proof}

Take $H\in C^2(M,\mathbb{R})$. Since the time-$1$ map of any tangent flow derived from a Hamiltonian vector field is measure preserving, we obtain a version of Oseledets' theorem (\cite{Oseledets}) for Hamiltonian systems. Thus, there exists a decomposition
$\mathcal{N}_x= \mathcal{N}^1_x \oplus \mathcal{N}^2_x\oplus \dots \oplus \mathcal{N}_{x}^{k(x)}$ called \emph{Oseledets splitting} and, for $1\le i\le k(x)\leq 2n$,
there are well defined real numbers
$$
\lambda_i(H,x)= \lim_{t\to\pm \infty} \frac1t \log \|\Phi_H^t(x) \cdot v_i\|,
	\quad \forall v_i \in E^i_x\setminus \{0\},
$$
called the \emph{Lyapunov exponents} associated to $H$ and $x$.
Since we are dealing with Hamiltonian systems (which imply the volume-preserving property),
we obtain that
\begin{equation}\label{eq.lyap1}
\sum_{i=1}^{k(x)} \lambda_i(H,x) =0.
\end{equation}
Notice that the spectrum of the symplectic linear map $\Phi^t_H$ is symmetric with respect to the $x$-axis and to the unit circle. In fact, if $\sigma\in\mathbb{C}$ is an eigenvalue with multiplicity $m$ so is $\sigma^{-1}$, $\overline{\sigma}$ and $\overline{\sigma}^{-1}$ keeping the same multiplicity (cf. Theorem~\ref{sympeigen}). Consequently, in the Hamiltonian context the Lyapunov exponents come in pairs and, for all $i\in\{1,...,n-1\}$, we have
\begin{equation}\label{eq.lyap2}
\lambda_i(H,x)=-\lambda_{2n-i-1}(H,x):=-\lambda_{\hat{i}}(H,x).
\end{equation}
Therefore, not counting the multiplicity and abreviating $\lambda(H,x)=\lambda(x)$, we have the increasing set of real numbers,
\begin{equation*}\label{eq.lyap3}
\lambda_1(x)\geq \lambda_2(x)\geq ... \geq \lambda_{n-1}(x) \geq 0 \geq -\lambda_{n-1}(x)\geq ... \geq -\lambda_2(x)\geq -\lambda_1(x),
\end{equation*}
or simply
\begin{equation*}\label{eq.lyap4}
\lambda_1(x)\geq \lambda_2(x)\geq ... \geq \lambda_{n-1}(x) \geq 0 \geq \lambda_{\hat{n-1}}(x)\geq ... \geq \lambda_{\hat{2}}(x)\geq \lambda_{\hat{1}}(x).
\end{equation*}
Associated to the Lyapunov exponents we have the Oseledets decomposition
\begin{equation}\label{eq.lyap5}
T_x M
	=\mathcal{N}^1(x)\oplus \mathcal{N}^2(x)\oplus ... \oplus \mathcal{N}^{n-1}(x) \oplus \mathcal{N}^{\hat{n-1}}(x)\oplus ... \oplus \mathcal{N}^{\hat{2}}(x)\oplus \mathcal{N}^{\hat{1}}(x).
\end{equation}

When all Lyapunov exponents are equal to zero, we say that the Oseledets splitting is trivial.
The vector field direction $\mathbb{R}X_{H}(x)$ is trivially an Oseledets's direction with zero Lyapunov exponent and its ``symplectic conjugate" is the direction transversal to the energy level.

\medskip

\begin{remark}
Let $p\in M$ be a closed orbit for $X^t_H$ of period $\pi$. Then, $\lambda_i=\pi^{-1}\log|\sigma_i|$ are the Lyapunov exponents, where $\sigma_i$ are the eigenvalues of $\Phi^{\pi}_H(p)$. Moreover, the Oseledets decomposition is defined by the eigendirections. Observe that eigenvalues can be complex and Lyapunov exponents are real numbers.
\end{remark}

\medskip

Define $\Lambda_i(x)=\lambda_1(x)+\lambda_{2}(x)+...+\lambda_{i}(x)$ which represents the top exponential growth of the $i$-dimensional volume corresponding to the evolving of $\Phi^t_H(x)$ (for details see \cite[\S 3.2.3]{LA}).

The main principle that makes the argument for the proof of our results possible is the following one due to Ma\~n\'e:

\medskip

\noindent\textbf{Ma\~n\'e principle:}
\emph{In the absence of a dominated splitting some perturbation of $\Phi^t_H$, by rotating its solutions, can be done in order to lower the Lyapunov exponents associated to the splitting without domination.}

\medskip

Actually, the ideas presented here are based on the now well-known Ma\~n\'e seminal ideas of mixing different Oseledets directions in order to decay its expansion rates and was deeply explored in \cite{Bessa, BR, MBJLD, BV2, SX, V}. This is the content of the following two lemmas. We observe that our notation with respect to the order of the Lyapunov exponents is inverted when compared to the one used in ~\cite{SX}, however, the proofs follows equally.
We recall that a splitting $E\oplus F$ has \emph{index} $k$ if $\dim(E)=k$. In our case the index is the dimension of the Oseledets subspace associated to the exponents $\lambda_1$,...,$\lambda_{i-1}$.

\begin{lemma}\label{SXia}
Let $H\in C^2(M,\mathbb{R})$, $x\in \mathcal{E}_{H,e}$ a hyperbolic periodic point for $X^t_H$ and $\lambda_{i-1}(H,x)-\lambda_i(H,x)>\delta$ where $\delta>0$. Moreover, we assume that $H$ does not have an $\ell$-dominated splitting of index $i-1$. Then, there exist $H_0$, such that $\|H-H_0\|<\epsilon({\ell})$ and $y \in {\mathcal{E}_{H_0,e_0}}$ a hyperbolic periodic point of $H_0$, arbitrarilly close to $x$, such that:
\begin{equation}\label{perturb}
\Lambda_{i-1}(H_0,y)<\Lambda_{i-1}(H,x)-\frac{\delta}{2}.
\end{equation}
\end{lemma}

\begin{proof}
The proof follows the same lines of the one in \cite[Proposition 9]{SX}. Let us recall the main steps: First, by \cite[Corollary 3.9]{BFR2}, we know that there is a residual set $\mathcal{R}$ in $C^2(M,\mathbb{R})$ such that, for any $H\in\mathcal{R}$, there is an open and dense set $\mathcal{S}(H)$ in $H(M)$ such that if $e\in\mathcal{S}(H)$ then any energy hypersurface of $H^{-1}(\{e\})$ is a homoclinic class. Actually, we can make a small perturbation on the Hamiltonian and on the energy in order to obtain that, given any hyperbolic periodic point $x$ of $H$, the set of its homoclinic related points, $\mathcal{H}_x$, is dense on $\mathcal{E}_{H,e}$. Moreover, we can do these perturbations arbitrarily small to guarantee that we still do not have $\ell$-dominated splitting of index $i-1$ for the analytic continuation of $x$.

Second, using the spectral gap hypothesis on $x$, i.e, $\lambda_{i-1}(H,x)-\lambda_i(H,x)>\delta$, we can spread this property to $\mathcal{H}_x$ by defining subbundles $E$ and $F$ of $\mathcal{N}_{\mathcal{H}_x}$, where $E$ is associated to Lyapunov exponents greater or equal than $\lambda_{i-1}$ and $F$ is associated to Lyapunov exponents less or equal than $\lambda_{i}$. Since the dominated splitting can be extended to the closure (see \cite{BDV}), if $E$ $\ell$-dominates $F$ in $\mathcal{H}_x$, then it can be extended to the whole energy hypersurface which is a contradiction.

Third, we use the lack of dominated splitting on $E\oplus F$ (say $E$ does not $\ell$-dominates $F$) to send directions in $E$ into directions in $F$ by small $C^2$ local perturbations along the segment of the orbit of a homoclinic point $z$. To put into operation Ma\~n\'e's principle we must use Theorem~\ref{Frank} several times. This will imply the desire inequalities (\ref{perturb}) for the homoclinic point $z$.

Finally, we just have to use Lemma~\ref{SXhamilt} to obtain a small perturbation that makes $z$ periodic.

\end{proof}

As an almost immediate consequence of Lemma~\ref{SXia}, we obtain (see ~\cite[Corollary 11]{SX}):

\begin{lemma}\label{SXia2}
Let $H\in C^2(M,\mathbb{R})$ and  $\mathcal{E}_{H,e}$ be an energy hypersurface without a dominated splitting of index $i-1$. Then, there exists $H_0$ arbitrarilly close to $H$ such that $H_0$ has a closed orbit $p$ with $\lambda_i(p)=\lambda_{i-1}(p)$.
\end{lemma}

Now we give the highlights of the proof since we follow closely \cite[\S 7]{SX}.

\begin{proof}
Let be given an open subset $U\subset M$ and let $H$ be a $C^2$-Hamiltonian with a far from partially hyperbolic re\-gu\-lar energy hypersurface intersecting $U$. We will prove that $H$ can be $C^{2}$-approximated by a $C^\infty$-Hamiltonian $H_0$ having a closed elliptic orbit through $U$.

By Remark ~\ref{Mane}, the existence of a dominated splitting implies partial hyperbolicity. Thus, if some energy hypersurface $\mathcal{E}_{H,e}$ intersects $U$ and is not   partially hyperbolic, then $\mathcal{E}_{H,e}$ does not have a dominated splitting at any fiber decomposition of the normal subbundle $\mathcal{N}$ that we consider.

Observe that $\epsilon$-close to $H$ we have that all systems have energy hypersurfaces far from being $\ell_\epsilon$-dominated. By contradiction, we assume that the systems is ``far" from having elliptic closed orbits, i.e., arbitrarily close to $H$ there are no elliptic closed orbits inside the intersection of a regular energy hypersurface and $U$. Thus, all closed orbits have some positive Lyapunov exponent $\lambda$.

Then, Lemma~\ref{SXia} is used several times to create a sequence of Hamiltonians $C^2$-converging to $H$ with a  Lyapunov exponent at the closed orbits passing throughout $U$ less than $r\lambda$, where $r\in(0,1)$ but close to $1$ which is a contradiction.
\end{proof}

\section{Weak shadowing (proof of Theorem~\ref{thm6})}

The next result says, in brief terms, that if a Hamiltonian system can be perturbed in order to create elliptic points, then for small
perturbations an iterate of the Poincar\'e map associated to the elliptic point is the identity. This prevents the
weak shadowing property.

\begin{proposition}\label{ML}
Let $(H,e, \mathcal{E}_{H,e})$ be a stably weakly shadowable Hamiltonian system.
Then, there exists a neighbourhood $\mathcal{V}$ of  $(H,e, \mathcal{E}_{H,e})$ such that
any $(H_0,e_0, \mathcal{E}_{H_0,e_0}) \in \mathcal{V}$ does not have elliptic points in $\mathcal{E}_{H_0,e_0}$.
\end{proposition}

\begin{proof}
Let $\mathcal{V}$ be a neighbourhood of  $(H,e, \mathcal{E}_{H,e})$ such that any Hamiltonian system in $\mathcal{V}$ is weakly shadowable.
By contradiction, let us assume that $(H_0,e_0, \mathcal{E}_{H_0,e_0}) \in \mathcal{V}$ has an elliptic point $q \in \mathcal{E}_{H_0,e_0}$ of period $\pi$. It follows from Lemma~\ref{continuum} and from the stability of elliptic points (see Remark~\ref{elliptic open}) that
there exists a Hamiltonian system $(H_1,e_1, \mathcal{E}_{H_1,e_1}) \in \mathcal{V}$ such that every point in a
$\xi$-neighborhood of $q$, in $\Sigma_q^c$, is a periodic point.
But, since in the current setting, $q$ is elliptic,
we have that $\Sigma_q^c=\Sigma_q$ and, therefore, as $f_1$ is linear, there exists
$m>0$ such that  $f_1^{m}$ is the identity map in a $\xi$-neighborhood of $q$. In order to simplify our arguments, let us suppose that
 $m=1$.
 Since $H_1$ has the weak shadowing property fixing $\epsilon \in (0,\frac{\xi}{4})$, there exists
 $\delta \in (0,\epsilon)$ and $T>0$ such that every $(\delta,T)$-pseudo-orbit $((x_i),(t_i))_{i \in \Z}$ is
 weakly $\epsilon$-shadowed by a trajectory $\mathcal{O}(z)$.

Let $x_0=q$.
Take $y\in \Sigma_q$ such that $d(q,y)=\frac{3\xi}{4}>2\epsilon$ and fix $\delta \in (0,\epsilon)$ sufficiently
small. We construct a bi-infinite sequence of points $((x_i), (t_i))_{i \in \Z}$ with $x_i \in \Sigma_q$
such that $((x_i), (t_i))_{i \in \Z}$ is a $(\delta,T)$-pseudo orbit
for some $T>0$.
For fixed $k \in \Z$, let $x_k=y$. There exist $x_i \in \Sigma_q$, $i \in \Z$ such that:
\begin{enumerate}
 \item[$\bullet$] $x_i=x_0$ and $t_i=\pi$ for $i \leq 0$;
\item[$\bullet$] $d(x_{i},x_{i-1}) < \delta$ and $t_i=\pi$ for $1 \leq i \leq k$;
\item[$\bullet$] $x_i=x_k$ and $t_i=\pi$ for $i >k$.
\end{enumerate}

Observe that we are considering that the return time at the transversal section is the same and equal to $\pi$. Clearly, it is not exactly equal to $\pi$, however
it is as close to $\pi$ as we want by just decreasing the $\xi$-neighborhood.
Therefore, $((x_i), (\pi))_{i \in \Z}$ is a $(\delta,T)$-pseudo-orbit for some $T>0$ such that $\pi \geq T$.

By the weakly shadowing property, there is a point $z \in \mathcal{E}_{H,e}$
such that $\{x_i\}_{i \in \Z} \subset \mathcal{B}_{\epsilon}(\mathcal{O}(z))$. Without loss of generality,
we may assume that $z \in \mathcal{B}(x_0,\epsilon)$.
Since $H_1$ is weakly shadowable, we have that for some $\tau=n\pi$,
$$
d(x_0,x_k) \leq d(x_0,z)+d(z,x_k) = d(x_0,z)+d((X_{H_1}^{\tau}(z),x_k)<2 \epsilon,
$$
which is a contradiction.

\end{proof}

Finally, the proof of Theorem ~\ref{thm6} is a consequence of Theorem~\ref{thm4} and Proposition~\ref{ML}.

\begin{proof}
Let $(H,e,\mathcal{E}_{H,e})$ be a stably weakly shadowable Hamiltonian system and suppose, by contradiction, that $\mathcal{E}_{H,e}$ is not partially hyperbolic.
Then, by Theorem~\ref{thm4}, there exists a $C^2$-close $C^\infty$-Hamiltonian $H_0$ with an elliptic closed orbit
on a nearby regular energy hypersurface $\mathcal{E}_{{H_0},{e_0}}$ and this contradicts Proposition~\ref{ML}.
\end{proof}

\section*{Acknowledgements}

JR was partially funded by European Regional Development Fund through the programme COMPETE and by the Portuguese Government through the FCT under the project PEst-C/MAT/UI0144/2011.

MJT was partially financed by FEDER Funds through ``Programa Operacional Factores de Competitividade - COMPETE'' and by Portuguese Funds through FCT - ``Funda\c{c}\~{a}o para a Ci\^{e}ncia e a Tecnologia'', within the Project PEst-C/MAT/UI0013/2011.

JR was partially supported by the FCT- ``Funda\c{c}\~{a}o para a Ci\^{e}ncia e a Tecnologia", project PTDC/MAT/099493/2008.

\end{section}

\end{document}